\documentclass[12pt]{amsart}
\usepackage{graphicx} % Required for inserting images
\usepackage{amsmath,amsfonts,amsthm,amssymb}
\usepackage{comment}
\usepackage{dsfont}
\usepackage{xcolor}
\usepackage{tabularx}
\usepackage{hyperref}
\usepackage{multirow}
\usepackage[
backend=biber,
style=alphabetic,
sorting=nty
]{biblatex}
\usepackage[margin=1in]{geometry}
\newtheorem{theorem}{Theorem}
\newtheorem{definition}[theorem]{Definition}
\newtheorem{proposition}[theorem]{Proposition}
\newtheorem{lemma}[theorem]{Lemma}
\newtheorem{corollary}[theorem]{Corollary}
\newtheorem{conjecture}[theorem]{Conjecture}
\newtheorem{question}[theorem]{Question}
\newtheorem{example}[theorem]{Example}
\theoremstyle{remark}
\newtheorem{remark}[theorem]{Remark}
\numberwithin{theorem}{section}

\title{Lie algebras in $\Ver_4^+$}
\author{Serina Hu}
\newcommand{\mcal}{\mathcal}
\newcommand{\mfrak}{\mathfrak}
\newcommand{\mbb}{\mathbb}
\newcommand{\1}{\mathds{1}}

\newcommand{\on}{\operatorname}
\newcommand{\inj}{\hookrightarrow}

\DeclareMathOperator{\Ver}{Ver}
\DeclareMathOperator{\Vect}{Vect}
\DeclareMathOperator{\End}{End}
\DeclareMathOperator{\Hom}{Hom}
\DeclareMathOperator{\im}{im}

\DeclareMathOperator{\id}{id}
\DeclareMathOperator{\Sym}{Sym}
\DeclareMathOperator{\Lie}{Lie}
\DeclareMathOperator{\GL}{GL}

\DeclareMathOperator{\Dist}{Dist}

\DeclareMathOperator{\gr}{gr}
\DeclareMathOperator{\FOLie}{FOLie}
\DeclareMathOperator{\coev}{coev}
\DeclareMathOperator{\diag}{diag}

\DeclareMathOperator{\ad}{ad}

\DeclareMathOperator{\Tr}{Tr}

\newcommand{\assign}{:=}

\usepackage{breqn}
\usepackage{makecell}

\begin{document}

\maketitle

\begin{abstract}
We develop Lie theory in the category $\Ver_4^+$ over a field of characteristic 2, the simplest tensor category which is not Frobenius exact, as a continuation of \cite{hu_supergroups_2024}. We provide a conceptual proof that an operadic Lie algebra in $\Ver_4^+$ is a Lie algebra, i.e. satisfies the PBW theorem, exactly when its invariants form a usual Lie algebra. % $\ker d$ is a classical Lie algebra. 
We then classify low-dimensional Lie algebras in $\Ver_4^+$, construct elements in the center of $U(\mfrak{gl}(X))$ for $X \in \Ver_4^+$, and study representations of $\mfrak{gl}(P)$, where $P$ is the indecomposable projective of $\Ver_4^+$.
\end{abstract}

\tableofcontents

\section{Introduction}

Let $k$ be an algebraically closed field of characteristic 2. The symmetric tensor category $\Ver_4^+$ is the tensor category of modules over the Hopf algebra $k[d]/d^2$ with $d$ primitive, with commutativity defined by the $R$-matrix $R:=1\otimes 1+d\otimes d$ (\cite{venkatesh_hilbert_2016}). \footnote{Note that this category has implicitly appeared much earlier than the referenced papers on Verlinde categories in the context of homotopy theory, see \cite{etingof_p-adic_2020}, Remark 3.4.} This category has one simple object, $\1$, with indecomposable projective cover $P = k[d]/d^2$.

In \cite{hu_supergroups_2024}, we studied the representation theory of general linear groups in $\Ver_4^+$. There we defined highest weights for representations of affine group schemes of finite type and classified the irreducible representations of general linear groups. In this paper, we develop some general theory around Lie algebras and representation theory of general linear Lie algebras in the category $\Ver_4^+$,  Since $\Ver_4^+$ is a reduction of the category of supervector spaces to characteristic $2$ (\cite{venkatesh_hilbert_2016}, Subsection 2.3), these Lie algebras may be viewed as Lie superalgebras in characteristic $2$, and more generally Lie theory in $\Ver_4^+$ may be regarded as a nontrivial version of superLie theory in characteristic $2$ (the trivial one being ordinary Lie theory, as there are no signs). More precisely, every object in $\Ver_4^+$ has the form $m \cdot \1 +nP$ where $P$ is the indecomposable projective, and $m\cdot \1+nP$ is the reduction to characteristic $2$ of $k^{(m+n|n)}$. 

We note that there are other nontrivial versions of superLie theory in characteristic $2$, 
for example the one discussed in \cite{bouarroudj_vectorial_2020} Subsection 1.2.3 and \cite{bouarroudj_classification_2023} Subsection 2.2. In this version, Lie superalgebras are defined as $\mathbb Z/2 \mathbb{Z}$-graded Lie algebras, with an additional structure (the squaring map), whereas in our setting Lie superalgebras have no $\mathbb Z/2 \mathbb{Z}$-grading, and instead have a derivation $d$ such that $d^2=0$. We discuss the relationship between these two theories in an upcoming paper.

The organization of the paper is as follows. In Section 2, we give a brief overview of Lie theory in symmetric tensor categories.

In Section 3, we prove that operadic Lie algebras in $\Ver_4^+$ are PBW iff $\ker d$ is a classical Lie algebra, i.e. $[x, x] = 0$ when $x \in \ker d$, giving a notion of Lie algebras in $\Ver_4^+$. This was proved computationally in \cite{kaufer_superalgebra_2018}, but our proof relies on the Koszul deformation principle.

In Section 4, we describe low-dimensional Lie algebras by describing a framework to classify Lie algebra structures on $m \cdot \1 + n P$. We then classify Lie algebra structures on $P$, $\1 + P$, $2 \cdot \1 + P$, and $2P$.

In Sections 5 and 6, we construct elements in the center of $U(\mfrak{gl}(X))$, $X \in \Ver_4^+$, by conjecturing an analog of the $p$-center and computing Casimir elements. We note that even for $X = P$, such elements do not generate the entire center of $U(\mfrak{gl}(P))$, indicating a need to develop this theory further.

In Section 7, we describe representations of $\mathfrak{gl}(P)$ and their restrictions to representations of $\GL(P)$.

\subsection{Acknowledgments}
I am deeply grateful to my advisor, Pavel Etingof, for both suggesting the problems in this paper and providing a huge amount of insightful advice on how to approach both the proofs and general ways of thinking about these problems. I am also grateful to Arun Kannan, Alex Sherman, and Kevin Coulembier for our helpful conversations about the Verlinde categories as well as Andrew Snowden and Karthik Ganapathy for asking many excellent questions about $\Ver_4^+$ during a seminar talk on this paper. This work was partially supported by NSF grant DMS-2001318. 

\section{Preliminaries}
% TODO: notation?
% define Ver_4^+, commutative algebra, group schemes and how to think about them (Hopf algebra or as functor of points (representable functor CommAlg -> Grp)), ind-algebras

In this section, we review preliminaries for working with Lie algebras in an arbitrary symmetric tensor category $\mcal{C}$, as described in \cite{venkatesh_harish-chandra_2022} and \cite{etingof_koszul_2018}.

We refer to \cite{hu_supergroups_2024} for an overview of symmetric tensor categories, the category $\Ver_4^+$, and its relationship to the category of supervector spaces. For notational ease, we will denote $dx$ by $x'$.

\begin{definition}
    The Lie operad $\mathbf{Lie} = \bigoplus_{n \ge 1} \mathbf{Lie}_n$ is the linear operad over $\mbb{Z}$ generated by an antisymmetric element $b \in \mathbf{Lie}(2)$ (the bracket) with defining relation
    \begin{equation*}
        b \circ (b \otimes 1) \circ (\id + (123) + (132)) = 0
    \end{equation*}
    in $\mathbf{Lie}(3)$ (the Jacobi identity). 
\end{definition}

\begin{definition}
An \emph{operadic Lie algebra} in a symmetric tensor category is an (ind)-object $L$ equipped with the structure of a $\mathbf{Lie}$ algebra. That is, it is equipped with a skew-symmetric bracket $\beta: L \otimes L \to L$, i.e.
\begin{equation*}
    \beta \circ (\id + (12)) = 0
\end{equation*}
such that $\beta$ satisfies the Jacobi identity, i.e.
\begin{equation*}
    \beta \circ (\beta \otimes \id) \circ (\id + (123) + (132)) = 0
\end{equation*}
where $(12)_{X \otimes Y} := c_{X \otimes Y}$ and $(123), (132)$ are the cyclic permutations on $X \otimes Y \otimes Z$ obtained by repeatedly applying the braiding. Let $\Lie(\mcal{C})$ be the category of operadic Lie algebras in $\mcal{C}$.
\end{definition}
\begin{remark}
    We call this an operadic Lie algebra instead of a ``Lie algebra'' because $L$ is not required to satisfy $[x, x]=0$, only $[x, y] = -[y, x]$, and these are not equivalent in characteristic 2, even in $\text{Vec}$. In particular, in characteristic 2 over vector spaces, an operadic Lie algebra will not satisfy the PBW theorem unless additionally $[x, x] = 0$.
\end{remark}

\begin{example}
    If $A$ is an associative algebra in $\mcal{C}$, it has a natural structure of an operadic Lie algebra via the commutator.
\end{example}

\begin{definition}
The \emph{universal enveloping algebra} $U(L)$ of an operadic Lie algebra $L$ is the (ind)-algebra $TL/(id - c - \beta)$, where $TL$ is the tensor algebra of $L$ and $c$ is the braiding of $\mcal{C}$..
\end{definition}

\begin{definition}
    The free operadic Lie algebra $\FOLie(V)$ associated to any object $V \in \mcal{C}$ is defined as follows: $\FOLie$ is the left adjoint to the forgetful functor from the category of operadic Lie algebras in $\mcal{C}$ to $\mcal{C}$; that is,
    \begin{equation*}
        \Hom_{\rm{Lie}}(\FOLie(V), L) = \Hom_{\mcal{C}}(V, L).
    \end{equation*}
    Concretely,
    \begin{equation*}
        \FOLie(V) = \bigoplus_{n \ge 1} \FOLie_n(V), \text{ } \FOLie_n(V) = (V^{\otimes n} \otimes \mathbf{Lie}_n)_{S_n}
    \end{equation*}
    where the subscript means the $S_n$-coinvariants.
\end{definition}

\begin{definition}
    Let $\phi^V: \FOLie(V) \to TV$ be the map of operadic Lie algebras induced by the degree 1 inclusion $V \inj TV$. (Note that $TV$ is an operadic Lie algebra because it is an associative algebra.) Let
    \begin{equation*}
        \phi^V = \bigoplus_{n \ge 1} \phi_n^V, \text{ where } \phi_n^V : \FOLie_n(V) \to V^{\otimes n}.
    \end{equation*}
    This map is not injective in general; let
    \begin{equation*}
        E_n(V) := \ker \phi_n^V.
    \end{equation*}
\end{definition}

Following the notation in \cite{etingof_koszul_2018}, we can define a Lie algebra as follows:

\begin{definition}
    We say that an operadic Lie algebra $L$ is a Lie algebra if the natural map $\beta^L: \FOLie(L) \to L$ induced by $\id: L \to L$ vanishes on $E_n(L)$.
\end{definition}

The vanishing of $\beta^L$ on $E_n(L)$ is a necessary condition for the PBW theorem to hold for $L$, and it is sufficient in Frobenius exact categories (\cite{etingof_koszul_2018}, Section 7).

\begin{example}
Affine group schemes have associated Lie algebras: $\mcal{O}(G)$ has an augmentation ideal $I = \ker \epsilon_G$, the counit of $\epsilon_G$. Its distribution algebra is the ind-object
\begin{equation*}
    \Dist(G) = \bigcup_{n = 0}^\infty (\mcal{O}(G)/I^n)^*,
\end{equation*}
and the Lie algebra $\mfrak{g}$ of $G$ is $(I/I^2)^* \subset \Dist(G)$.
\end{example}

We will prove in Section \ref{pbw_theorem} that operadic Lie algebras in $\Ver_4^+$ satisfying $[x, x]= 0 $ when $dx = 0$ are Lie algebras in this sense. That is, an operadic Lie algebra in $\Ver_4^+$ is a Lie algebra exactly when $\ker d$ is a Lie algebra in $\text{Vec}$. We also show that the PBW theorem holds for an operadic Lie algebra $L$ under the above condition, i.e. the natural map $\eta: \Sym L \to \gr U(L)$ is an isomorphism.

Determining whether the PBW theorem holds for operadic Lie algebras in a general symmetric tensor category is rather involved and discussed in \cite{etingof_koszul_2018}, but it is much easier if the underlying object is Koszul.

Recall the definition of the Koszul complex of an object $X$ in a symmetric tensor category as given in \cite{etingof_koszul_2018}.
\begin{definition}
  Let $\Lambda^2 X$ be the image of $c_{X, X} - 1$ and $S X = T X / (\Lambda^2 X)$.
\end{definition}

\begin{definition}
  Let $\mathbb{S}^2 X$ be the kernel of $c_{X, X} - 1$, or $(S^2
  X^{*})^{*}$, and $\wedge X = T X /(\mathbb{S}^2 X)$.
\end{definition}

\begin{definition}
  Let $\Lambda^n X = (\wedge^n X^{*})^{*}$, i.e. which can be identified with the sign-twisted $S_n$ invariants $\bigcap \operatorname{im} c_i - 1 \subset X^{\otimes n}$, where $c_i$ is the braiding swapping the $i$th and $i+1$th factors.
\end{definition}

There are multiplication maps $\mu : X \otimes S^j X \rightarrow S^{j + 1} X$
and $X^{*} \otimes \wedge^{i - 1} X^{*} \rightarrow \wedge^i X^{*}$;
dualizing the second gives us a map $\mu_{*} : X^{*} \otimes \Lambda^i X
\rightarrow \Lambda^{i - 1} X$ (evaluation). 
\begin{definition}[\cite{etingof_koszul_2018}]
  The Koszul complex of an object $X$ is $K^{\bullet} (X) = S X \otimes
  \Lambda^{\bullet} X$, and the differential is given by
  \begin{equation*}
  \partial : S X \otimes \Lambda^i X \xrightarrow{\coev_X} X \otimes
     X^{*} \otimes S X \otimes \Lambda^i X \xrightarrow{P_{2, 3}} X \otimes
     S X \otimes X^{*} \otimes \Lambda^i X \xrightarrow{\mu \otimes
     \mu_{*}} S X \otimes \Lambda^{i - 1} X.
     \end{equation*}
     If the Koszul complex of $X$ is exact, we say that $X$ is Koszul.
\end{definition}

By Theorem 4.11 in \cite{etingof_koszul_2018}, if the underlying object of an operadic Lie algebra $L$ is Koszul, the following are equivalent:
\begin{itemize}
    \item $L$ is PBW
    \item $L$ is a Lie algebra
    \item $\beta^L$ vanishes on $E_2(L)$ and $E_3(L)$.
\end{itemize}
We show that all objects in $\Ver_4^+$ are Koszul, hence we can use this theorem to characterize Lie algebras in $\Ver_4^+$.

\section{Lie algebras in $\Ver_4^+$ and the PBW theorem}\label{pbw_theorem}

In this section, we prove the PBW theorem for any operadic Lie algebra $\mfrak{g}$ in $\Ver_4^+$ satisfying $[x, x] = 0$ when $x' = 0$ for $x \in \mfrak{g}$.

First, we show that $P$ is Koszul. 
\begin{proposition}
\label{p_koszul}
    $P$ has exact Koszul complex, so it is Koszul.
\end{proposition}

\begin{proof}    
Let $P$ have basis $x, x'$. Then
\begin{align*}
    \Lambda^2 P &= \text{span}(x' \otimes x', x \otimes x' + x' \otimes x), \\
    S P &= k[x, x'], (x')^2 = 0, \\
    \mathbb{S}^2 P &= \text{span}(x' \otimes x', x \otimes x' + x' \otimes x).
\end{align*}
Therefore, $\wedge P \cong S P$ and $\wedge^n P$ is spanned by $x^{\otimes n}$ and $x^{\otimes n - 1} \otimes x'$. Then $P^*$ has a dual basis $y', y$ with $y'(x) = 1$, $y(x') = 1$, and $dy = y'$. Moreover, $\wedge^n (P^*)$ is spanned by $y^{\otimes n}$ and $y^{\otimes n - 1} \otimes y'$. Thus $\Lambda^n P$ is also two-dimensional and spanned by
\begin{align*}
    \mathbf{t}_n &:= (y^{\otimes n})^*, \\
    \mathbf{x}_n &:= (y^{\otimes n - 1} \otimes y')^*.
\end{align*}
When $X = P$,
$\mu_{*}$ acts as
\begin{align*}
  x^{*} \otimes \mathbf{t}_i & \mapsto 0\\
  x^{*} \otimes \mathbf{x}_i & \mapsto \mathbf{t}_{i - 1}\\
  x^{\prime *} \otimes \mathbf{t}_i & \mapsto \mathbf{t}_{i - 1}\\
  x^{\prime *} \otimes \mathbf{x}_i & \mapsto \mathbf{x}_{i - 1}
\end{align*}

\begin{lemma}
  The differential $\partial : S^j P \otimes \Lambda^i P \rightarrow S^{j + 1} P \otimes
  \Lambda^{i - 1} P$ acts as
  \begin{equation*}
  x^j \otimes (\alpha\mathbf{t}_i + \beta\mathbf{x}_i) + x^{j - 1} x' \otimes
     (\gamma\mathbf{t}_i + \delta\mathbf{x}_i) \mapsto x^{j + 1} \otimes
     \beta\mathbf{t}_{i - 1} + x^j x' \otimes ((\alpha + \delta) \mathbf{t}_{i - 1} +
     \beta\mathbf{x}_{i - 1}) .
     \end{equation*}
\end{lemma}

\begin{proof}
  Let $f(x) \in S^j P$. For $X = P$, coevalution acts as $1 \mapsto 1 \otimes 1^{*} + d \otimes
  d^*$, so the composition $P_{2, 3} \circ \coev_P$ acts as
  \begin{align*}
  f (x) \otimes \Lambda^i P &\mapsto (1 \otimes 1^{*} + d \otimes
     d^*) \otimes f (x) \otimes \Lambda^i P\\
     &\mapsto (1 \otimes f (x)
     \otimes 1^{*} + d \otimes f (x) \otimes d^* + d \otimes d f (x)
     \otimes 1^{*}) \otimes \Lambda^i P.
     \end{align*}
  Applying $\mu$, we get
  \begin{equation*}
  (xf (x) \otimes 1^{*} + x'f (x) \otimes d^* + x' f(x)' \otimes
     1^{*}) \otimes \Lambda^i P
     \end{equation*}
  and the third summand is $0$ because $f(x)'$ will be a multiple of
  $x'$, so multiplication by $x'$ will kill it. Since $\partial : S^j X \otimes
  \Lambda^i X \rightarrow S^{j + 1} X \otimes \Lambda^{i - 1} X$, we may look only at homogeneous components. Applying $\mu_{*}$, we get that
  $\partial$ sends a (symmetric) degree $j$ element $x^j \otimes (\alpha\mathbf{t}_i + \beta\mathbf{x}_i) + x^{j -
  1} x' \otimes (\gamma\mathbf{t}_i + \delta\mathbf{x}_i)$ to
  \begin{equation*}
  x^{j + 1} \otimes \beta\mathbf{t}_{i - 1} + x^j x' \otimes ((\alpha + \delta)
     \mathbf{t}_{i - 1} + \beta\mathbf{x}_{i - 1}) .
     \end{equation*}
\end{proof}

Therefore, the image of $\partial$ (on $S^j P \otimes \Lambda^i P$) is the
  subspace where the $x^{j + 1} \otimes \mathbf{x}_{i - 1}$ coefficient is 0
  and the $x^{j + 1} \otimes \mathbf{t}_{i - 1}$ and $x^j x' \otimes
  \mathbf{x}_{i - 1}$ coefficients are equal. On the other hand, the kernel
  of $\partial : S^{j + 1} P \otimes \Lambda^{i + 1} P \rightarrow S^{j + 2} P
  \otimes \Lambda^i P$ is the subspace where the $x^{j + 1} \otimes
  \mathbf{x}_{i - 1}$ coefficient is 0 and $\alpha + \delta = 0$, i.e. $\alpha = \delta$, so the
  coefficients of $x^{j + 1} \otimes \mathbf{t}_{i - 1}$ and $x^j x'
  \otimes \mathbf{x}_{i - 1}$ must be equal. So this complex is exact in
  positive degrees.
  
  In degree 0, since $\mathbf{x}_0 = 0 , \mathbf{t}_0 = 1$, the
  differential $\partial : K^1 (P) \rightarrow K^0 (P)$ acts as
  \begin{equation*} x^j \otimes (\alpha\mathbf{t}_1 + \beta\mathbf{x}_1) + x^{j - 1} x' \otimes
     (\gamma\mathbf{t}_1 + \delta\mathbf{x}_1) \mapsto x^{j + 1} \otimes \beta + x^j x'
     \otimes (\alpha + \delta) . \end{equation*}
  Hence the image of this morphism is the space of polynomials
  with no constant term. So $H_0$ is the quotient of $S P$ by $\im \partial$, which is
  $\1$. Therefore, $P$ is Koszul.
\end{proof}

\begin{corollary}
    Since $\1$ is Koszul as well, all objects in $\Ver_4^+$ are Koszul.
\end{corollary}

Hence, we obtain the following condition for being a Lie algebra, which was proved computationally in \cite{kaufer_superalgebra_2018}.
\begin{corollary}
    An operadic Lie algebra is a Lie algebra satisfying the PBW theorem if and only if $[x, x] = 0$ for all $x \in \ker d$.
\end{corollary}
\begin{proof}
By Theorem 4.11 in \cite{etingof_koszul_2018}, an operadic Lie algebra whose underlying object is Koszul is PBW iff it is a Lie algebra iff $\beta^L$ vanishes on $E_2(L)$ and $E_3(L)$. Since $\operatorname{char} k = 2$, $E_3(L) = 0$. Meanwhile, $E_2(L)$ is the first Frobenius twist of $L$; thus if $L = m \cdot \1 + nP$, $E_2(m \cdot \1 + nP) = m \cdot \1$, as the Frobenius twist of $P$ is $0$. Hence, $\beta$ vanishes on $E_2(L)$ exactly when $[x, x] = 0$ for $x \in m \cdot \1 \subset L$.

Note that for $x \in nP$, the skew-symmetry condition is $[x, x] + [x, x] + [x', x'] = 0$, so for $y \in \im d \subset nP$, $[y, y] = 0$. Therefore, $m \cdot \1 + nP$ is a Lie algebra exactly when $[x, x] = 0$ for $x \in \ker d$.
\end{proof}
    
\section{Determination of Lie algebras of low dimensionalities}

In this section, we provide a general framework for classifying Lie algebra structures in $\Ver_4^+$, then classify Lie algebras in $\Ver_4^+$ with underlying objects $P$, $\1 + P$, $2 \cdot \1 + P$, and $2P$ up to isomorphism.

We first deduce some basic properties of Lie algebras in $\Ver_4^+$. Recall that we require the bracket to satisfy the following:
\begin{itemize}
    \item $[x, y] + [y, x] + [y', x'] = 0$ (skew-symmetry);
    \item $[[x, y], z] + [[z, x], y] + [[y, z], x] + [[x', y], z'] + [[z', x], y'] + [[y', z], x'] = 0$ (Jacobi identity);
    \item $[x, x] = 0$ if $x' = 0$ (so that the operadic Lie algebra is a Lie algebra).
\end{itemize}

In particular, this implies that $[x', x'] = 0$ and that $[x, y] = [y, x]$ for $y \in \ker d$. Therefore, a Lie algebra with underlying object $m \cdot \1$ is a classical $m$-dimensional Lie algebra.

Because $d$ is primitive, $d$ is a derivation: we have 
\begin{equation*}
    d([x, y]) = [x', y] + [x, y'].
\end{equation*}
In particular, $[x, x], [x, x'] \in \ker d$.

\begin{definition}
    We say a Lie algebra $\mfrak{g}$ in $\Ver_4^+$ is alternating if $[x, x] = 0$ for all $x \in \mfrak{g}$. We say $\mfrak{g}$ is skew-symmetric if $[x, y] = [y, x]$ for all $x, y \in \mfrak{g}$, i.e. $[x', y'] = 0$ and $\im d$ is an abelian Lie algebra.
\end{definition}
\begin{remark}
    Every alternating Lie algebra is skew-symmetric, since
    $$[x, y] + [y, x] = [x + y, x + y] - [x, x] - [y, y].$$
\end{remark}

\begin{proposition}
    Every Lie algebra structure on $n \cdot \1 + P$ is skew-symmetric. 
\end{proposition}
\begin{proof}
    Since $\im d$ is one-dimensional, it must be the trivial Lie algebra.
\end{proof}

\subsection{Classification of ordinary Lie algebra structures of dimension up to 4}

Here we review the well-known classification of ordinary Lie algebras of dimension at most 3, which will be used in the later subsections. We also state the classification in the dimension 4 case to cover all the 4-dimensional cases in $\Ver_4^+$.

The only 1-dimensional Lie algebra is $\1$ with the trivial bracket.

The following proposition is a well-known classical result for 2-dimensional Lie algebras:
\begin{proposition}\label{p_lie_algebras}
    If the underlying object of $L$ is $2 \cdot \1$ with basis $x, y$, then $L$ is isomorphic to one of the following:
    \begin{itemize}
        \item an abelian Lie algebra
        \item $[x,y] = x$, which we denote by $L_2$
    \end{itemize}
\end{proposition}

The classification of Lie algebras with underlying object $3 \cdot \1$ is described in e.g. \cite{jacobson_lie_1979} Section 1.4:
\begin{proposition}
    Let $L$ be a Lie algebra with underlying object $3 \cdot \1$ with basis $x, y, z$. If $[L, L] \ne L$, $L$ is isomorphic to one of
    \begin{itemize}
        \item abelian 3-dimensional Lie algebra
        \item $[x, y] = z$, other brackets are 0
        \item $[x, y] = x$, other brackets are 0
        \item $[x, y] = 0$, $[x, z] = x$, $[y, z] = \lambda y$, $\lambda \ne 0$
        \item $[x, y] = 0$, $[x, z] = x + y$, $[y, z] = y$
    \end{itemize}
    If $[L, L] = L$, then isomorphism classes of $L$ correspond to equivalence classes of 3 by 3 symmetric matrices $A$ where $A \sim B$ if $B = \rho N^T A N$ where $\rho \ne 0$ and $N$ is invertible.
\end{proposition}

The classification of 4-dimensional Lie algebras is described in \cite{strade_lie_2006} in the non-solvable case and \cite{graaf_solvable_2005} in the solvable case. We will not reproduce it here because it its not used elsewhere in this paper.

\subsection{Classification of Lie algebra structures on $P$}\label{lie_p_classification}

\begin{proposition}
    If the underlying object of $L$ is $P$ with basis $x, x'$, then $L$ is isomorphic to one of
    \begin{itemize}
        \item an abelian Lie algebra, which we denote by $P_a$;
        \item the Lie algebra with $[x, x] = 0, [x', x] = x'$, which we denote by $P_s$;
        \item the Lie algebra with $[x, x] = x', [x', x] = 0$, which we denote by $P_n$;
    \end{itemize}
    and these are pairwise non-isomorphic.
\end{proposition}
\begin{proof}
    We know that $[x, x], [x, x'] \in \ker d$; thus they are both multiples of $x'$. The Jacobi identity tells us that
    \begin{itemize}
        \item $[[x, x], x] + [[x', x], x'] = 0$
        \item $[[x, x], x'] = 0$
    \end{itemize}
    but the second is vacuous and the first reduces to $[[x, x], x] = 0$.

    If $L$ is not abelian, then we have two cases:
    \begin{itemize}
        \item If $[x, x] = 0$, then $[x, x'] = \lambda x' \ne 0$. By scaling $x$ by $\lambda^{-1}$, we get an isomorphism to the second case in the proposition.
        \item If $[x, x] = \lambda x' \ne 0$, then $[[x, x], x] = 0$ implies $\lambda[x, x'] = 0$. Hence $[x, x'] = 0$. Again, by scaling $x$ by $\lambda^{-1}$, we get an isomorphism to the third case.
    \end{itemize}
    These are not isomorphic to each other: in $P_s$, $\ad x$ is semisimple, but in $P_n$, $\ad x$ is nilpotent, which motivates the notation. Moreover, any isomorphism in $\GL(P)(k)$ will scale $[x, x]$ and $[x', x]$, so both $P_s$ and $P_n$ degenerate to the abelian case.
\end{proof}

\begin{remark}
    We see that $P_s$ does not degenerate to $P_n$ because $P_s$ is alternating while $P_n$ is not. Neither does $P_n$ degenerate to $P_s$, since in $P_n$, $\ad y$ is nilpotent for all $y \in P_n$, while in $P_s$ this is not the case.
\end{remark}

\subsection{Classification of Lie algebra structures on $m \cdot \1 + nP$} \label{gl_mnp_framework}

In this section, we describe a general framework for classifying Lie algebras with underlying object $m \cdot \1 + nP$. Let $\mfrak{g} := m \cdot \1 + nP$. Then $\ker d \subset \mfrak{g}$ is an ordinary $m + n$-dimensional Lie algebra $L$ (though the following analysis works for any $\ker d \supset L \supset \im d$) and $\mfrak{g}/L$ has the structure of an $L$-module. Choose a splitting $\mfrak{g} = V \oplus L$ as vector spaces; this induces an $L$-module structure on $V$, which we denote by $V_f$. Then $d: V \to L$ is an injective $L$-module homomorphism with image $\im d$. In particular, if $\mfrak{g} = nP$, then $d$ is an isomorphism and $V$ is isomorphic to the adjoint representation. Otherwise, $V$ is a subrepresentation of the adjoint representation. 

For the rest of this section, elements of $L$ are denoted by $z, w$ and elements of $V$ are denoted by $v, x, y$. Write
$$[x, w] = f(w)x + A(w)x$$
where $f: L \to \End(V)$ and $A: L \to \Hom(V, L)$. Then $f$ is the $L$-module structure on $V$.
Write
$$[x, y] = B(x \otimes y) + C(x \otimes y)$$
where $B: V \otimes V \to V$ and $C: V \otimes V \to L$. Skew-symmetry implies that 
\begin{align*}
    B(x \otimes y) = B(y \otimes x), B(x \otimes x) = 0 \\
    C(x \otimes y) + C(y \otimes x) + [y', x'] = 0.
\end{align*}

The Jacobi identity on $x, z, w$ for $z, w \in L$ and $x \in V$ implies that $f$ is a representation of $L$ on $V$, as stated above, and that 
\begin{equation}\label{a_derivation}[A(z)x, w] + A(w)f(z)x + [A(w)x, z] + A(z)f(w)x + A([z, w])x = 0,\end{equation} so $A$ is a derivation $L \to \Hom_{\Vect}(V_f, L)$. 

Any automorphism $\varphi: V \to V$ leaves $f$ unchanged and takes $A(w) \mapsto A(w) \circ \varphi$, $B \mapsto \varphi^{-1} B (\varphi \otimes \varphi)$, and $C \mapsto C(\varphi \otimes \varphi)$.

We can reduce choosing a different splitting of $\mfrak{g}$ to translating $V$ by a linear map $\varphi: V \to L$ because we have covered automorphisms of $V$ above. Such a translation leaves $f$ unchanged, and we have
\begin{align*}
    A(w, (1+\varphi)x) &= A(w, x) + [w, \varphi(x)] + \varphi(f(w)x) \\
    B((1 +\varphi)x, (1+\varphi)y)&= B(x \otimes y) + f(\varphi(y))x + f(\varphi(x))y \\
    C((1+\varphi)x, (1+\varphi)y) &= C(x \otimes y) + A(\varphi(y))x + A(\varphi(x))y + \varphi(B(x \otimes y)).
\end{align*}
In particular, if $D$ is the differential in the Chevalley-Eilenberg complex, $A$ is shifted by $D\varphi: L \to \Hom(V_f, L)$. Therefore, $A$ corresponds to a well-defined class in $H^1(L, \Hom(V_f, L))$ up to automorphisms of $V_f$.

Now we determine how properties of the bracket correspond to relations that $f, A, B, C$ must satisfy.
For $x, y \in V$, we have $[x, y]' = [x', y] + [x, y'] \in \ker d$, i.e. $$f(x')y + A(y)x' + f(y')x + A(x)y' \in \ker d.$$ Therefore, $$f(x')y = f(y')x, B(x \otimes y)' = A(x')y + A(y')x.$$ Note that the former is implied by $f$ being a subrepresentation of the adjoint representation.

For $x \in V$, $z \in L$, $$[x', z] = (f(z)x)'.$$ For $x \in V$, $z \in L$, the Jacobi identity on $x, x, z$ simplifies to $[[x, x], z] + [[x', z], x'] = 0$. Therefore, $$[C(x \otimes x), z] = (f(x')f(z)x)'.$$ For $x, y \in V$, the Jacobi identity simplifies to $[[x, x], y] + [[x', y], x'] = 0$ when $x, y \in V$. Hence, $$f(C(x \otimes x))y + A(C(x \otimes x))y + f(x')f(x')y + A(x')f(x')y + [A(x')y, x']= 0,$$ so
\begin{align*}
    f(C(x \otimes x))y &= f(x')f(x')y \\
    A(C(x \otimes x))y &= A(x')f(x')y + [x', A(x')].
\end{align*}

For $x, y \in V$, $z \in L$, the Jacobi identity simplifies to $$[[x, y], z] + [[z, x], y] + [[y, z], x] + [[y', z], x'] = 0.$$
Therefore,
\begin{align*}
    &f(z)B(x \otimes y) + A(z)B(x \otimes y)\\
    + &B(f(z)x \otimes y) +  C(f(z)x \otimes y) + f(A(z)x)y + A(A(z)x)y \\
    + &B(f(z)y \otimes x) + C(f(z)y \otimes x) + f(A(z)y)x + A(A(z)y)x \\
    + &[[y', z], x'] + [C(x \otimes y), z] = 0.
\end{align*}
Gathering terms in $L$, we need
\begin{gather*}
    f(z)B(x \otimes y) +B(f(z)x \otimes y) + B(f(z)y \otimes x) + f(A(z)x)y + f(A(z)y)x = 0 \\
    \implies [B, z](x \otimes y) = f(A(z)x)y + f(A(z)y)x,
\end{gather*}
and likewise gathering terms in $V$,
$$[C, z](x \otimes y) = A(z)B(x \otimes y) + A(A(z)x)y + A(A(z)y)x.$$

Finally, the Jacobi identity for $v, x, y \in V$ states that
\begin{align*}
    &B(B(x \otimes y) \otimes v) + f(C(x \otimes y))v + f(v')f(x')y \\
    + &B(B(v \otimes x) \otimes y) + f(C(v \otimes x))y + f(y')f(v')x \\
    + &B(B(y \otimes v) \otimes x) + f(C(y \otimes v))x + f(x')f(y')v = 0
\end{align*}
and
\begin{align*}
    &C(B(x \otimes y) \otimes v) + A(C(x \otimes y))v + [A(x')y, v'] + A(v')f(x')y \\
    + &C(B(v \otimes x) \otimes y) + A(C(v \otimes x))y + [A(v')x, y'] + A(y')f(v')x \\
    + &C(B(y \otimes v) \otimes x) + A(C(y \otimes v))x + [A(y')v, x'] + A(x')f(y')v.
\end{align*}

\begin{theorem}\label{lie_m_np_classification}
    Isomorphism classes of Lie algebra structures on $m \cdot \1 + nP$, correspond to tuples $(L, f, A, B, C)$, where
    \begin{enumerate}
        \item $L$ is an ordinary $m + n$-dimensional Lie algebra;
        \item $f$ is an $n$-dimensional subrepresentation of the adjoint representation of $L$, denoted $V_f$; let $d: V_f \subset L$ denote this inclusion and write $d(v) = v'$ for convenience;
        \item $A$ is a derivation $L \to \Hom_{\Vect}(V_f, L)$ whose class in $H^1(L, \Hom_{\Vect}(V_f, L))$ is well-defined;
        \item $B: V \otimes V \to V$, is a symmetric map with $B(x \otimes x) = 0$, up to translation by $f(\id \otimes \varphi)(1 + \sigma)$, where $\varphi: V \to L$ is any linear map (corresponding to choosing a different splitting of $\mfrak{g}$) and $\sigma$ is swapping factors;
        \item $C: V \otimes V \to L$, is a map such that $C(x \otimes y) = C(y \otimes x) + [y', x']$ up to translation by $A(\id \otimes \varphi)(1 + \sigma) + \varphi \circ B$, with $\varphi, \sigma$ as above,
    \end{enumerate}
    satisfying 
    \begin{align*}
        % f(x')y &= f(y')x, x, y \in V \\
        B(x \otimes y)' &= A(x')y + A(y')x, x, y \in V \\
        [C(x \otimes x), z] &= (f(x')f(z)x)', x \in V, z \in L \\
        f(C(x \otimes x)) &= f(x')f(x'), x \in V \\
        A(C(x \otimes x)) &= A(x')f(x') + (\ad x')A(x'), x \in V \\
        [B, z](x \otimes y) &= f(A(z)x)y + f(A(z)y)x, x,y \in V, z \in L \\
        [C, z](x \otimes y) &= A(z)B(x \otimes y) + A(A(z)x)y + A(A(z)y)x, x, y \in V, z \in L \\
        0 &= (B(B \otimes \id) + f(C \otimes \id) + (f \circ d)((f \circ d) \otimes \id))\sigma_3 \\
        0 &= C(B \otimes \id) + A(C \otimes \id) + (\ad \circ d)((A \circ d) \otimes \id) + df(A(d \otimes \id) \otimes \id))\sigma_3
    \end{align*}
    where $\sigma_3 = \on{Id} + (123) + (132)$ and the last two relations are on $V \otimes V \otimes V$.
    An automorphism $\varphi: V \to V$ takes $$(f, A, B, C) \mapsto (f, A(\cdot)\varphi, \varphi^{-1}B(\varphi \otimes \varphi), C(\varphi \otimes \varphi)).$$
\end{theorem}

We simplify this data in two special cases of this classification: when $\ker d$ is abelian and when $n = 1$.
\begin{proposition}
If $\ker d = L \subset \mfrak{g}$ is abelian, then the data of a Lie algebra structure on $m \cdot \1 + nP$ can be simplified to:
\begin{enumerate}
    \item An $n$-dimensional Lie algebra $\mfrak{g}_0$;
    \item an $m+n$-dimensional $\mfrak{g}_0$-module $L$;
    \item an injective 1-cocycle $d: \mfrak{g}_0 \to L$;
    \item a class $C \in H^2_e(\mfrak{g}_0, L)$ such that $p(C) \in L^{\mfrak{g}_0}$. Here $H^2_e$ is the extended cohomology -- coboundaries for this are the same as in the ordinary case, but cocycles are symmetric rather than skew-symmetric forms satisfying the 2-cocycle condition -- and $p: H^2_e(\mfrak{g}_0, L) \to Z^1(\mfrak{g}_0, L)^{(1)}$ is the map $c(x, y) \mapsto c(x, x)$.
\end{enumerate}
\end{proposition}
\begin{proof}
If $L$ is abelian, then $f = 0$ and $\mfrak{g}_0 := \mfrak{g}/L$ is a Lie algebra. We can rewrite $A$ as a map $\rho: \mfrak{g}_0 \to \End(L)$ so that $L$ has the structure of an $\mfrak{g}_0$-module. The Jacobi identity implies that $B$ induces a Lie bracket on $\mfrak{a}$ and that
\begin{align*}
    &C(B(x \otimes y) \otimes v) + \rho(z)C(x \otimes y) \\
    + &C(B(v \otimes x) \otimes y) + \rho(y)C(v \otimes x) \\
    + &C(B(y \otimes v) \otimes x) + \rho(x)C(y \otimes v).
\end{align*}
Also, since $L$ is abelian, $C$ is symmetric. Therefore, $C$ is an extended $2$-cocycle on $\mfrak{g}_0$ taking values in $L$. Since $B$ is the bracket on $\mfrak{g}_0$, choosing a different splitting of $\mfrak{g}$ translates $C$ by a 2-coboundary. Hence, $C \in H^2_e(\mfrak{g}_0, L)$. We have a linear map $p: H^2_e(\mfrak{g}_0, L) \to Z^1(\mfrak{g}_0, L)^{(1)}$, the Frobenius twist of the 1-cocycles, given by $c(x, y) \mapsto c(x, x)$. 
Since $C(x \otimes x) \in \ker \rho$, we then require $p(C) \in L^{\mfrak{g}_0}$. Finally, that $B(x \otimes y)' = \rho(y)x' + \rho(x')y$ is equivalent to $d: \mfrak{g}_0 \to L$ being an injective 1-cocycle.
\end{proof}

\begin{proposition}
    Isomorphism classes of Lie algebra structures on $m \cdot \1 + P$ correspond to tuples $(L, f, A, b, c)$, where
    \begin{enumerate}
        \item $L$ is an ordinary $m + 1$-dimensional Lie algebra;
        \item $f$ is a one-dimensional Lie algebra representation of $L$ on $k$, denoted $k_f$;
        \item $A$ is a class in $H^1(L, L \otimes k_f)$;
        \item $b, c$ are elements in $L$ such that $b$ is central, $c \ne 0$ and $c \in \ker f$, and $Ab = f(Ac)c$;\footnote{$Ab$ is well-defined as $b$ is central, and while $Ac$ is not well-defined, $A^* f$ is well-defined, so $f(Ac)$ is well-defined.}
    \end{enumerate}
    up to scaling $(L, f, A, b, c) \mapsto (L, f, tA, t^2b, tc)$.

    If $f = 0$, then $A \in H^1(L, L)$ and $c \in L$ is also central, so an isomorphism class of Lie algebra structures corresponds to a tuple $(L, A, b, c)$ with $A \in H^1(L, L)$, $b, c \in L$ central elements with $c \ne 0$ and $Ab = 0$, up to rescaling as above.

    If $f \ne 0$, then an isomorphism class of Lie algebra structures corresponds to a tuple $(L_0, D, A_0, a, b, c, q)$ where
\begin{enumerate}
    \item $L_0$ is an $m$-dimensional Lie algebra with central elements $b, c$, $c \ne 0$;
    \item $D$ is a class in $H^1(L_0, L_0)$ such that $Dc = c$ and $Db = 0$. Choosing a representative $\mathbf{D}$ of $D$, we get a Lie algebra $k \cdot p \ltimes L_0$ where $p$ acts by $\mathbf{D}$ on $L_0$;
    \item $a: L_0 \to k$ is a character such that $a(b) = 0$, $a \circ D = a$, and $a \wedge D \in H^2(L_0, L_0)$ vanishes;
    \item $A_0 \in \End(L_0)/\ad(L_0)$ such that $d(A_0) = a \wedge D$, $(D-1)(A_0) = 0 \in H^1(L_0, L_0)$, and $A_0b = a(c)c$; \footnote{$d(A_0)$ is well-defined as $A_0$ is defined up to inner derivations, and $(D-1)(A_0)$ is well-defined as a class in $H^1(L_0, L_0)$ because we check that $d((D-1)A_0) = 0$. $A_0b$ is well-defined because $b$ is central.}
    \item choosing a representative $\mathbf{A}_
    0$ of $A_0$, an element $q \in L_0$ such that $(\mathbf{D}-1)(\mathbf{A}_0) = \ad(q)$
\end{enumerate}
up to rescaling as above.
\end{proposition}
\begin{proof}
This is the special case $n = 1$ of the previous theorem. In this situation, $V$ is one-dimensional, so $f: L \to k$ is a character of $L$ corresponding to a one-dimensional subrepresentation $k_f$ of $L$ and $A \in H^1(L, L \otimes k_f)$. Let the injection $d: k_f \subset L$ be spanned by $c \in L$, $c \in \ker f$. Choose a representative $\mathbf{A}$ of $A$. Also, $B = 0$, and $C$ is a central element $b \in L$, with $\mathbf{A}b = f(\mathbf{A}(c))c$ and $b \in \ker f$. We see that while $b$ is invariant under translations, under automorphisms of $k_f$, i.e. scaling, $A \mapsto tA$, $c \mapsto tc$, and $b \mapsto t^2 b$.

If $f = 0$, then $A \in H^1(L, L)$ and $c \in L$ is central. Hence, we can write that $Ab = f(Ac)c$, as $H^1(L, L)$ acts on $Z(L) = H^0(L, L)$.

If $f \ne 0$, then $L_0 := \ker f$ is an ideal in $L$ with codimension $1$, and $b, c \in L_0$ are central elements in $L_0$. Fix $p \in L$ with $f(p) = 1$; then $\mathbf{D} := \ad p: L_0 \to L_0$ is a derivation of $L_0$ with $\mathbf{D}c = c$. We recall that $A$ is defined up to inner derivations, so by translating by the appropriate multiple of $p$ we may choose a representative $\mathbf{A}$ of $A$ such that $\mathbf{A}p \in L_0$.

Now for $x \in L_0$, write $\mathbf{A}x = \mathbf{A}_0x + a(x)p$ for $\mathbf{A}_0: L_0 \to L_0$ and $a: L_0 \to k$. That $\mathbf{A}$ is a derivation implies that $a$ is a character and 
$$\mathbf{A}_0[z, x] = [\mathbf{A}_0z, x] + [z, \mathbf{A}_0x] + a(z)\mathbf{D}x + a(x)\mathbf{D}z.$$ 
For $y \in k_f$ and $z \in L$, translating $y \mapsto y + z$ sends $A_0 \mapsto \ad z$ and leaves the others unchanged; translating $p \mapsto p + z$ sends $\mathbf{D} \mapsto \mathbf{D} + \ad z$, $\mathbf{A}_0 + a \otimes z$, and $\mathbf{A}p \mapsto \mathbf{A}p + \mathbf{A}_0 z a(z)p$; we can then translate $y$ by $a(z)p$ to get
\begin{enumerate}
    \item $\mathbf{A} \mapsto \mathbf{A} + \ad z$
    \item $\mathbf{A}_0 + a \otimes z + a(z)\mathbf{D}$
    \item $\mathbf{A}p \mapsto \mathbf{A}p + \mathbf{A}_0z$.
\end{enumerate}
Also, $\mathbf{A}b = f(\mathbf{A}c)c$ implies that $a(b) = 0$ and $\mathbf{A}_0b = a(c)c$. Finally, by writing out $\mathbf{A}[p, x]$ for $x \in L_0$, we see that $a \circ \mathbf{D} = a$ and
$$[\mathbf{D}, \mathbf{A}_0] = \mathbf{A}_0 + \ad \mathbf{A}p.$$

Cohomologically, we can therefore interpret this data as follows:
\begin{enumerate}
    \item We have a class $D := [\mathbf{D}] \in H^1(L_0, L_0)$, a character $a: L_0 \to k$, and central elements $b, c \in L_0$ such that $a \circ D = a$, $Db = 0$, $Dc = c$, $a(b) = 0$, $c \ne 0$, and the class $C_1(a, D) := a \wedge D \in H^2(L_0, L_0)$ vanishes ("first obstruction to Lie algebra structure).
    \item We can define a torsor $T = T(a, D)$ over $H^1(L_0, L_0)$ as follows: pick $p \in L$ such that $f(p) = 1$; then $\mathbf{D} = \ad p$ is a representative of $D$. Let $T$ be the set of $A_0 \in \End(L_0)/\ad(L_0)$ such that $d(A_0) = a \wedge D$. This torsor depends only on $D$ up to canonical isomorphism, since shifting $p$ by $z$ shifts $T$ by $a \otimes z$. 
    \item We have another class $(D-1)(A_0) \in H^1(L_0, L_0)$, since $$d((D-1)(A_0)) = (D-1)d(A_0) = (D-1)(a \wedge D) = 0.$$ This class is well-defined regardless of the representative we choose for $D$, since translating $\mathbf{D}$ by $\ad z$ translates $\mathbf{A}_0$ by $a \otimes z - a(z)\mathbf{D}$, hence translates $(\mathbf{D}-1)(\mathbf{A}_0)$ by $\ad(\mathbf{A}_0 z)$. This class must vanish since $(\mathbf{D}-1)(\mathbf{A}_0) = \ad Ap$ ("second obstruction to Lie algebra structure").
    \item We can then define another torsor $M = M(a, D, A_0)$ over $Z(L_0)$ as the set of elements $q \in L_0$ such that $(\mathbf{D}-1)(\mathbf{A}_0) = \ad(q)$ for representatives $\mathbf{D}, \mathbf{A}
    _0$ of $D, A_0$ respectively. As before, $M$ depends only on $D$ and not $p$ up to canonical isomorphism, since shifting $p$ by $z$ shifts $M$ by $\mathbf{A}_0z$.
\end{enumerate}

Hence, we have the following data:
\begin{enumerate}
    \item A Lie algebra $L_0$ with central elements $b, c$, $c \ne 0$.
    \item A class $D \in H^1(L_0, L_0)$ such that $Dc = c$ and $Db = 0$; choosing a representative $\mathbf{D}$ of $D$, we get a Lie algebra $k \cdot p \ltimes L_0$ where $p$ acts by $\mathbf{D}$ on $L_0$.
    \item A character $a: L_0 \to k$ such that $a(b) = 0$, $a \circ D = a$, and $a \wedge D \in H^2(L_0, L_0)$ vanishes.
    \item A map $\mathbf{A}_0 \in \End(L_0)$ such that its equivalence class $A_0 \in \End(L_0)/\ad(L_0)$ satisfies $d(A_0) = a \wedge D$, $(D-1)[A_0] = 0 \in H^1(L_0, L_0)$, and $A_0(b) = a(c)c$
    \item An element $q \in L_0$ such that $(\mathbf{D}-1)(\mathbf{A}_0)= \ad(q)$, corresponding to $q = \mathbf{A}p$.
\end{enumerate}
\end{proof}
\begin{remark}
    In the $f \ne 0$ case, when $L_0$ is abelian, the data on $D$ and $A_0$ can be simplified to conditions on linear maps:
    \begin{enumerate}
        \item We have $D \in \End(L_0)$ and $a \wedge D: L_0 \wedge L_0 \to L_0$ vanishes.
        \item We have $A_0 \in \End(L_0)$ and $(D-1)A_0 = 0$; $q$ is unrestricted.
    \end{enumerate}
    We still require $Da = a$, $Dc = c$, $Db = 0$, and $A_0(b) = a(c)c$.
\end{remark}

In the case that $L$ is 1-dimensional, corresponding to $\mfrak{g} = P$, we see that we get the same classification as worked out by hand in Section \ref{lie_p_classification}:
\begin{example}
    If $L$ is 1-dimensional, it is abelian. Then $A: L \to L$ is a scalar. 
    \begin{enumerate}
        \item If $A \ne 0$, WLOG $A = 1$, and then $b = 0$. So in this case $[y, y] = 0$ and $[y, y'] = y'$, which is $P_s$.
        \item If $A = 0$, then:
        \begin{enumerate}
            \item If $b = 0$, then $\mfrak{g}$ is abelian, which is $P_a$.
            \item If $b \ne 0$, then we can scale so $b = c$ since $y \mapsto ty$ sends $(A, b, y') \mapsto (tA, t^2 b, ty')$. Then $[y, y] = y'$ and all other brackets are 0, which is $P_n$.
        \end{enumerate}
    \end{enumerate}
\end{example}
(There are no cases when $f \ne 0$ since that requires $L$ to be at least 2-dimensional.)

\subsection{Classification of Lie algebra structures on $\1 + P$}

We now use the general framework to classify isomorphism classes of Lie algebra structures on low-dimensional cases, beginning with $\1 + P$.
\begin{proposition}
\label{1p_lie}
    If the underlying object of $L$ is $\1 + P$ with basis $x, y, y'$, $L$ is isomorphic to exactly one of the following Lie algebras, which are organized by their subalgebras and quotients:

    \begin{center}
        \begin{tabular}{| c | c | c | c | c | c | c |}
            & $[x, y']$ & $[y, y]$ & $[y', y]$ & $[x, y]$ & subalgebras & quotients \\
            \hline
            1 & $0$ & $0$ & $0$ & $0$ & $\1, \1^2, P_a$ & $\1, \1^2, P_a$ \\ \hline
            2 & $y'$ & $0$ & $0$ & $y + \lambda y'$ & $L_2$, $P_a$, $P_s$ & $L_2$ \\ \hline
            3 & $0$ & $y'$ & $0$ & $\lambda x$ & $\1^2$, $P_n$ & $P_n$, $L_2$, $\1$ \\ \hline
            4 & $0$ & $y'$ & $0$ & $y'$ & $\1^2$, $P_n$ & $\1^2$, $\1$ \\ \hline
            5 & $0$ & $x$ & $\lambda x$ & $0$ & $\1^2$ & $P_a$, $\1$ \\ \hline
            6 & $0$ & $x$ & $\lambda x + y'$ & $0$ & $\1^2$ & $\1$, $P_s$, $\1^2$ if $\lambda = 0$, $P_n$ if $\lambda \ne 0$ \\ \hline
            7 & $0$ & $0$ & $0$ & $x$ & $\1^2$, $P_a$ & $P_a$, $L_2$, $\1$ \\ \hline
            8 & $0$ & $0$ & $0$ & $y'$ & $\1^2$, $P_a$ & $\1^2$, $\1$ \\ \hline
            9 & $0$ & $0$ & $y'$ & $x + y'$ & $\1^2$, $P_s$ & $L_2$, $\1$ \\ \hline
            10 & $0$ & $0$ & $y'$ & $\lambda x$ & $\1^2$, $P_s$ & $L_2$, $P_s$, $\1$ \\ \hline
            11 & $0$ & $0$ & $x$ & $x + \lambda y'$ & $\1^2$ & $\1$, $P_s$, $P_a$ if $\lambda = 0$ \\ \hline
            12 & $0$ & $0$ & $x$ & $y'$ & $\1^2$ & $\1, P_s$ \\ \hline
            13 & $0$ & $0$ & $x$ & $0$ & $\1^2, \1$ & $P_a$ \\ \hline
        \end{tabular}
    \end{center}
    Here the subalgebras column indicates which Lie algebra structures on $P$ and $\1^2$ appear as subalgebras in $L$, and likewise for quotients, and $\lambda$ is a parameter taking arbitrary values in $k$.
    
\end{proposition}
\begin{proof}
First suppose $f = 0$. Then $L$ is 2-dimensional and $y' \ne 0$ is central, so $L$ must have nontrivial center. Hence, $L$ is abelian. We casework on whether or not $b = 0$.
    \begin{enumerate}
    \item If $b = 0$, then $A: L \to L$ can be any linear map. Then we have several cases, depending on the Jordan normal form of $A$.
    \begin{enumerate}
        \item If $A = 0$, then $\mfrak{g}$ is abelian (row 1).
        \item if $A$ is nonzero and diagonalizable, after a change of basis it is of the form $\diag(1, \mu)$.
        \begin{enumerate}
            \item If $Ay' = 0$, we get that $[y, y'] = 0$ and $[y, x] = x$ (after changing basis to $y', Ay'$) (row 7).
            \item If $Ay' = \mu y'$ for $\mu \ne 0$, we can scale so that $Ay' = y'$, in which case we get that $[y', y'] = y'$, $[y', x] = \mu x$ (row 10).
            \item If $Ay'$ is independent of $y'$, in which case $\mu \ne 1$, we can change basis to $Ay', y'$, then scale $A$ so that $[y, Ay'] = Ay' + \lambda y'$ for some $\lambda$ (row 11). We can achieve all $\lambda$ by choosing $A$ appropriately.
        \end{enumerate}
        \item If $A$ is nilpotent, we have two cases:
        \begin{enumerate}
            \item If $Ay' = 0$, then we can change basis so that $[y, x] = y'$, which is the only nonzero bracket (row 8).
            \item If $Ay' \ne 0$, then changing basis so that $x \mapsto Ay'$, we get $[y, y'] = x$ and all other brackets are 0 (row 13).
        \end{enumerate}
        \item If $A$ is not diagonalizable but invertible, we can change basis and rescale so that $A = \begin{pmatrix} 1 & 1 \\ 0 & 1 \end{pmatrix}$. If $Ay' = y'$, then we get the case $[y, y'] = y'$ and $[y, x] = x + y'$ (row 9). If $Ay'$ is independent of $y'$, then we can set $x = Ay'$ and scale appropriately to get $[y, y'] = x$ and $[y, x] = y'$ (row 12).
    \end{enumerate}
    \item If $b \ne 0$, then $A$ must kill $b$. We have the following cases:
    \begin{enumerate}
        \item $A = 0$, so $[y, y]$ is the only nonzero bracket. If $b$ is a multiple of $y'$, we can scale so that $[y, y] = y'$ (row 3 with $\lambda = 0$). If $b$ is not a multiple of $y'$, then we can change basis so $[y, y] = x$ (row 5 with $\lambda = 0$).
        \item $A$ is nilpotent and $Ab = 0$; after change of basis, $A$ is of the form $\begin{pmatrix} 0 & 1 \\ 0 & 0 \end{pmatrix}$. If $b$ is a multiple of $y'$, then we can rescale so $[y, y] = y'$ and $[y, x] = y'$ (row 4). If $b$ is not a multiple of $y'$, then we can transform so $[y, y] = x$, and $[y, y'] = \lambda x$ and can choose $A$ appropriately to get all nonzero values of $\lambda$ (row 5 with $\lambda \ne 0$).
        \item $A$ is diagonalizable, with one nonzero eigenvector $v$ with eigenvalue we can scale to $1$. If $b$ is a multiple of $y'$, we can transform so that $[y, y] = y'$ and $[y, x] = \lambda x$, choosing $A$ appropriately to get all $\lambda$ (row 10). If $y' = v$ is the other eigenvector, we have $[y, y'] = y'$ and $[y, y] = x$, while $[y, x] = 0$ (row 6 with $\lambda = 0$). If both $b$ and $v$ are independent of $y'$, then we can transform so that $[y, y] = x$, $[y, y'] = \lambda x + y'$, and $[y, x] = 0$, and can choose $A$ appropriately to get all nonzero $\lambda$ (row 6 with $\lambda \ne 0$).
    \end{enumerate}
    \end{enumerate}
    
Now suppose $f \ne 0$, so $L_0$ is one-dimensional.
Then the requirement that $Dc = c$ fixes $D$ to be the identity map, while $b = 0$. Since $a \wedge D = 0$, we have $a = 0$; since $(D-1)A_0 = A_0$, we have $A_0 = 0$ also. Meanwhile, $q \in L_0$ is free. Hence, there is only one Lie algebra structure with $f \ne 0$; it has $[x, y'] = y'$ and $[y, x] = y + \lambda y'$ (row 2).
\end{proof}

\begin{remark}
    The Lie algebras described in rows 2 and 7-13 are alternating.
\end{remark}
\begin{remark}
    Let $v, w$ be a basis for $P$, with $v' = w$. The Lie algebra $\mfrak{sl}(P)$ of traceless elements of $\mfrak{gl}(P)$ is spanned by $y := v \otimes w^*$, $y' := v \otimes v^* + w \otimes w^*$, and $x := w \otimes v^*$. We can check that $[y, y] = y'$, $[y, x] = y'$, and $[y', y] = 0$, so this is row 3.

    The Lie algebra $\mfrak{pgl}(P) := \mfrak{gl}(P)/I$ is spanned by (the equivalence clasess of) $y := v \otimes v^*$, $y' := w \otimes v^*$, and $x := v \otimes w^*$. We can check that $[y, y] = 0$, $[x, y] = x$, and $[y', y] = y'$. Hence, row 10 with $\lambda = 1$ is $\mfrak{pgl}(P)$.
\end{remark}

\subsection{Classification of Lie algebra structures on $2 \cdot \1 + P$}

In this section, we classify Lie algebra structures on $2 \cdot \1 + P$. Let $v, v'$ be a basis for $P$ and $x, y$ be a basis for $2 \cdot \1$. When $L$ or $L_0$ is not abelian, we use $\mathbf{A}$ for a linear map and $A = [\mathbf{A}]$ for its cohomology class.

First suppose that $f = 0$. Then $L$ is a 3-dimensional Lie algebra with nontrivial center; say $L$ has basis $x, y, z$. Either $L$ is abelian, $[x, y] = z$ is the only nonzero bracket, or $[x, y] = x$ is the only nonzero bracket. Let $A_{xx}$ be the $x$-coefficient of $Ax$ and $A_{yy}$ be the $y$-coefficient of $Ay$ with respect to the basis $x, y, z$. 

\begin{proposition}
    Suppose that $f = 0$ and $L$ is isomorphic to the 3-dimensional Lie algebra with $[x, y] = z$ the only nonzero bracket. The isomorphism classes of such Lie algebra structures are the following, where each one has $[x, y]=v'$:
\begin{center}
    \begin{tabular}{| c | c | c | c | c |}
        & $[v, v]$ & $[v, v']$ & $[v, x]$ & $[v, y]$ \\
        \hline
        1.1 & $0$ & $0$ & $0$ & $\lambda x$ \\ \hline
        1.2 & $0$ & $0$ & $x$ & $\lambda x + y$ \\ \hline
        1.3 & $0$ & $v'$ & $\lambda x$ & $(\lambda + 1)y$ \\ \hline
        1.4 & $v'$ & $0$ & $0$ & $\lambda x$ \\ \hline
        1.5 & $\lambda v'$ & $0$ & $x$ & $\mu x + y$ \\ \hline
    \end{tabular}
\end{center}
\end{proposition}
\begin{proof}
Since $z$ spans the center of $L$, $c$ is a nonzero multiple of $z$, and we can scale so that $c = z$. Now $b$ is a central element in the kernel of $\mathbf{A}$, so $b = \lambda z$. Moreover, in order for $\mathbf{A}$ to be a derivation, we must have that $\mathbf{A}z$ is central and that $$\mathbf{A}z = [\mathbf{A}x, y] + [x, \mathbf{A}y].$$ We have two cases:
\begin{enumerate}
    \item If $b \ne 0$, i.e. $\lambda \ne 0$, then $\mathbf{A}z = 0$. For $\mathbf{A}$ to be a derivation, the only relation $\mathbf{A}$ must satisfy now is $$\mathbf{A}z = [\mathbf{A}x, y] + [x, \mathbf{A}y].$$ Hence $\mathbf{A}_{xx} = \mathbf{A}_{yy}$ and $A$ is traceless. Then we have the following cases for $\mathbf{A}$, some of which result in the same cohomology class as they differ by an inner derivation:
    \begin{enumerate}
        \item $\mathbf{A} = 0$, in which case $[v, v] = \lambda v'$ is the only nonzero bracket; then we can scale so that $[v, v] = v'$ (row 4 when $\lambda = 0$).
        \item $\mathbf{A}$ is nilpotent and is a single Jordan block. Hence, with change of basis we can set $\mathbf{A}x = z$ and $\mathbf{A}y = \mu x$. Adjusting by $\ad y$, which sends $x \mapsto z$, we can make it so that $\mathbf{A}$ has instead two Jordan blocks with the Jordan block of $z$ one-dimensional, so $\mathbf{A}$ is in the same class as in the next case.
        \item $\mathbf{A}$ is nilpotent, has two Jordan blocks, and the Jordan block of $z$ is one-dimensional. Then with change of basis we can set $\mathbf{A}x = 0$ and $\mathbf{A}y = \mu x, \mu \ne 0$. Hence, $[v, x] = 0$, $[v, y] = \mu x$, $[v, v'] = 0$, and $[v, v] = \lambda v'$. We can rescale so that $[v, v] = v'$ (by sending $v \mapsto \lambda^{-1}v$ and $x \mapsto \lambda^{-1}x$) (row 4 when $\lambda \ne 0$).
        \item $\mathbf{A}$ is nilpotent, has two Jordan blocks, and the Jordan block of $z$ is two-dimensional. Then with change of basis we have $\mathbf{A}x = 0$ and $\mathbf{A}y = z$. Adjusting by $\ad x$, which sends $y \mapsto z$, we see that $[\mathbf{A}] = 0$.
        \item $\mathbf{A}$ is not nilpotent: then we can scale $\mathbf{A}$ to have nonzero generalized eigenvalue $1$ and with change of basis we can set $\mathbf{A}x = x$ and $\mathbf{A}y = \mu x + y$. So $[v, x] = x$, $[v, y] = \mu x + y$, $[v, v'] = 0$, and $[v, v] = \lambda v'$ (row 5).
    \end{enumerate}
    \item If $b = 0$, then $\mathbf{A}z$ is not necessarily $0$, but it is still a multiple of $z$. If $\mathbf{A}z = 0$, then the cases for $\mathbf{A}$ are the same as above, except that $[v, v] = 0$ instead of $v'$ (rows 1 and 2). If $\mathbf{A}z \ne 0$, we can rescale $\mathbf{A}$ so that $\mathbf{A}z = z$. Then we must have $\mathbf{A}_{xx} = \mathbf{A}_{yy} + 1$, so $\mathbf{A}$ is traceless. We have two cases, which both correspond to the same cohomology class:
    \begin{enumerate}
        \item If $\mathbf{A}$ has generalized eigenvalues $\lambda, \lambda + 1, 1$ with $\lambda \ne 1, 0$, then $\mathbf{A}$ is diagonalizable. So we have $\mathbf{A}x = \lambda x$ and $\mathbf{A}y = (\lambda + 1)y$. This gives the Lie algebra with $[v, x] = \lambda x$, $[v, y] = (\lambda + 1)y$, and $[v, v'] = v'$ (row 3).
        \item If $\mathbf{A}$ is not diagonalizable, after change of basis it will be of the form $\begin{pmatrix} 1 & \lambda & 0 \\ 0 & 1 & 0 \\ 0 & 0 & 0 \end{pmatrix}$. Adjusting by $\lambda \ad y$, which sends $x \mapsto \lambda z$, we get a diagonalizable matrix, so $[\mathbf{A}]$ is the same as above.
    \end{enumerate}
\end{enumerate}
\end{proof}

\begin{proposition}
    Suppose $f = 0$ and $L$ is isomorphic to the 3-dimensional Lie algebra with $[x, y] = x$ the only nonzero bracket. The isomorphism classes of Lie algebra structure satisfying the above are:
\begin{center}
    \begin{tabular}{| c | c | c | c | c |}
        & $[v, v]$ & $[v, v']$ & $[v, x]$ & $[v, y]$ \\
        \hline
        2.1 & $0$ & $0$ & $0$ & $v'$ \\ \hline
        2.2 & $0$ & $0$ & $0$ & $0$ \\ \hline
        2.3 & $0$ & $0$ & $0$ & $y$ \\ \hline
        2.4 & $0$ & $v'$ & $0$ & $\lambda v'$ \\ \hline
        2.5 & $0$ & $v'$ & $0$ & $\lambda y$ \\ \hline
        2.6 & $\lambda v'$ & $0$ & $0$ & $v'$ \\ \hline
        2.7 & $v'$ & $0$ & $0$ & $0$ \\ \hline
        2.8 & $\lambda v'$ & $0$ & $0$ & $y$ \\ \hline
    \end{tabular}
\end{center}
\end{proposition}
\begin{proof}
As in the $[x, y] = z$ case, $z$ spans the center of $L$, so $c$ is a nonzero multiple of $z$, and we can scale so that $c = z$, while $b$ is a multiple of $z$. The condition that $\mathbf{A}$ is a derivation is equivalent to $\mathbf{A}z$ being central, so $\mathbf{A}z = \lambda z$, and that $\mathbf{A}x = [\mathbf{A}x, y] + [x, \mathbf{A}y]$, which reduces to $\mathbf{A}x = \mu x$ and $\mathbf{A}_{yy} = 0$. Hence, both $x$ and $z$ are eigenvectors for $\mathbf{A}$. We can add $\mu \ad y$ to $\mathbf{A}$ to get $\mathbf{A}x = 0$.

If $b \ne 0$, then $\mathbf{A}z = 0$, so $\mathbf{A}$ is determined by where it sends $y$; because $\ad x: y \mapsto x$, we can make it so that $\mathbf{A}_{yx} = 0$. Then
\begin{enumerate}
    \item if $\mathbf{A}_{yy} = 0$, then we can scale $A$ so that $Ay = z$ (row 6) or $Ay = 0$. If $Ay = 0$, we can scale $A$ so that $[v, v] = v'$ (row 7).
    \item If $\mathbf{A}_{yy} = \lambda \ne 0$, we can scale so that $\mathbf{A}_{yy} = 1$, and changing basis via $y \mapsto Ay$ gives $\mathbf{A}y = y$ (row 8).
\end{enumerate}

If $b = 0$, then
\begin{enumerate}
    \item If $\mathbf{A}z = 0$, the cases are the same as above (rows 1-3).
    \item If $\mathbf{A}z \ne 0$, we can scale $\mathbf{A}$ so that $\mathbf{A}z = z$. Then we can no longer scale $\mathbf{A}$ so that a coefficient in $\mathbf{A}y$ is one, so either $\mathbf{A}y = \lambda z$ (row 4) or $\mathbf{A}y = \lambda y$ (row 5).
\end{enumerate}
\end{proof}

\begin{proposition}
    Suppose $f = 0$ and $L$ is abelian. If $b = 0$, then the isomorphism classes of such Lie algebras are
    \begin{center}
    \begin{tabular}{| c | c | c | c | c | c |}
        & $[v, v]$ & $[v, v']$ & $[v, x]$ & $[v, y]$ & JNF of $A$ \\
        \hline
        3.1 & $0$ & $0$ & $0$ & $0$ & $0$\\ \hline
        3.2 & $0$ & $v' + x + y$ & \multirow{3}{*}{$\lambda x$} & \multirow{3}{*}{$\mu y$} & \multirow{3}{*}{$\diag(1, \lambda, \mu)$\footnote{If $\lambda = \mu$, then $x, y$ are indistinguishable, so either $Av' = v' + y$ or $Av' = v'$ and row 2 is redundant. If $\lambda = \mu = 1$, we have only the case that $Av' = v'$ and both rows 2 and 3 are redundant.}} \\ \cline{1-3}
        3.3 & $0$ & $v' + y$ & & & \\ \cline{1-3}
        3.4 & $0$ & $v'$ & & & \\ \hline
        3.5 & $0$ & $0$ & $x$ & $\lambda y$ & $\diag(0, 1, \lambda)$ \\ \hline
        3.6 & $0$ & $v' + x + \lambda y$ & $x$ & $x + y$ & \multirow{3}{*}{$\begin{pmatrix} 1 & 1 & 0 \\ 0 & 1 & 1 \\ 0 & 0 & 1 \end{pmatrix}$} \\ \cline{1-5}
        3.7 & $0$ & $v' + x$ & $x$ & $v' + y + \lambda x$ & \\ \cline{1-5}
        3.8 & $0$ & $v'$ & $x + v'$ & $x + y$ & \\ \hline
        3.9 & $0$ & $\alpha x + y$ & $0$ & $x$ & \multirow{3}{*}{$\begin{pmatrix} 0 & 1 & 0 \\ 0 & 0 & 1 \\ 0 & 0 & 0 \end{pmatrix}$} \\ \cline{1-5}
        3.10 & $0$ & $x$ & $0$ & $v'$ & \\ \cline{1-5}
        3.11 & $0$ & $0$ & $v'$ & $x$ & \\ \hline
        3.12 & $0$ & $\lambda v' + (\lambda + 1)(x + y)$ & $x$ & $x + y$ & \multirow{4}{*}{$\begin{pmatrix} 1 & 1 & 0 \\ 0 & 1 & 0 \\ 0 & 0 & \lambda \end{pmatrix}$\footnote{Note that if $\lambda = 1$, we have only three cases: $v' = x$ (so $Av' = v'$ and $Ay = v' + y$, row 15), $v' = y$ (so $Av' = v' + x$, row 14), or $v' = z$ (so $Av' = v'$ and $Ax = x + y$, row 13).}} \\ \cline{1-5}
        3.13 & $0$ & $v' + (\lambda + 1)x$ & $x$ & $x + y$ & \\ \cline{1-5}
        3.14 & $0$ & $v' + x$ & $x$ & $\lambda y$ & \\ \cline{1-5}
        3.15 & $0$ & $v'$ & $v' + x$ & $\lambda y$ & \\ \hline
        3.16 & $0$ & $v' + x + y$ & $0$ & $x$ & \multirow{4}{*}{$\begin{pmatrix}0 & 1 & 0 \\ 0 & 0 & 0 \\ 0 & 0 & 1 \end{pmatrix}$} \\ \cline{1-5}
        3.17 & $0$ & $v' + x$ & $0$ & $x$ & \\ \cline{1-5}
        3.18 & $0$ & $x$ & $0$ & $y$ & \\ \cline{1-5}
        3.19 & $0$ & $0$ & $v'$ & $y$ & \\ \hline
        3.20 & $0$ & $x$ & $0$ & $0$ & \multirow{3}{*}{$\begin{pmatrix}0 & 1 & 0 \\ 0 & 0 & 0 \\ 0 & 0 & 0 \end{pmatrix}$} \\ \cline{1-5}
        3.21 & $0$ & $0$ & $0$ & $x$ & \\ \cline{1-5}
        3.22 & $0$ & $0$ & $v'$ & $0$ & \\ \hline

    \end{tabular}
    \end{center}
\end{proposition}
\begin{proof}
Since $[v, v] = b = 0$ and $L$ is abelian, $A$ can be any linear map $L \to L$.
\begin{enumerate}
    \item If $A = 0$, then $\mfrak{g}$ is abelian (row 1).
    \item If $A$ is nonzero and diagonalizable:
    \begin{enumerate}
        \item If $Av' \ne 0$, then suppose $x, y, z$ is an eigenbasis for $A$ where $v' = x + a y + b z$ and $x$ has nonzero eigenvalue (here it is possible for $a = b = 0$, in which case $v'$ is an eigenvector). We can scale $A$ so that $Ax = x$, while $Ay = \lambda y, Az = \mu z$. Then $Av' = v' + \alpha y + \beta z$. Note that $y$ and $z$ can be scaled, so either $Av' = v' + y + z$ (row 2), $Av' = v' + z$ (row 3), or $Av' = v'$ (row 4). (Additionally, if $\lambda = \mu$, then $y, z$ are indistinguishable, so either $Av' = v' + z$ or $Av' = v'$. If $\lambda = \mu = 1$, we have only the case that $Av' = v'$.)
        \item If $Av' = 0$, then there is another eigenvector $x$ with nonzero eigenvalue. By scaling $A$, we can ensure that $Ax = x$. So we have $Av' = 0$, $Ax = x$, and $Ay = \lambda y$ (row 5).
    \end{enumerate}
    \item If $A$ is a single Jordan block:
    \begin{enumerate}
        \item If $A$ has nonzero eigenvalue, then after rescaling and a change of basis it is of the form $\begin{pmatrix} 1 & 1 & 0 \\ 0 & 1 & 1 \\ 0 & 0 & 1 \end{pmatrix}$ for basis $x, y, z$. We can check that if $v' = \alpha x + \beta y + \gamma z$, then $Av' = v' + \beta x + \gamma y$. If $\gamma \ne 0$, then $x, y, v'$ is a basis for $A$ where $Ax = x$, $Ay = x + y$, and $Av' = v' + \beta x + \gamma y$ (row 6). Note that $\beta$ and $\gamma$ scale together and $\gamma \ne 0$, so we can scale so that $\gamma = 1$. If $\gamma = 0$ but $\beta \ne 0$, then (after scaling $y$) $x, v', z$ is a basis where $Ax = x$, $Av' = v' + \beta x$, and $Az = z + v' + \lambda x$ (row 7). Again, note that $\beta$ and $\lambda$ scale together and $\beta \ne 0$, so we can set $\beta = 1$. If $\gamma = \beta = 0$, then we can scale so that $Av' = v'$, $Ay = y + v'$, and $Az = z + y$ (after relabeling $y, z$ to $x, y$ respectively, row 8). 
        % If $v'$ is a multiple of $x$, we can scale so that $Av' = v'$, $Ay = y + v'$, $Az = z + y$ (after relabeling $y, z$ to $x, y$ respectively, row 6). If $v'$ is a multiple of $y$, we can scale $x, z$ so that $Ax = x$, $Av' = v' + x$, $Az = z + v'$ (after relabeling $z$ to $y$, row 7). If $v'$ is a multiple of $z$, we can scale so that $Ax = x$, $Ay = x + y$, $Av' = y + v'$ (row 8).
        \item If $A$ is nilpotent, then after a change of basis it is of the form $\begin{pmatrix} 0 & 1 & 0 \\ 0 & 0 & 1 \\ 0 & 0 & 0 \end{pmatrix}$, with basis $x, y, z$. We can shift $y \mapsto y + \lambda x$ and $z \mapsto z + \lambda y$, so we can make it so that either $v'$ is a multiple of $x$ or $v'$ is in the span of $y, z$. If $v'$ has nonzero $z$-coefficient, then after scaling we have $v' = \alpha y + z$, so $Ax = 0$, $Ay = x$, and $Av' = \alpha x + y$ (row 9). If $v'$ is a multiple of $y$, then after scaling we have $Ax = 0$, $Av' = x$, $Az = v'$ (row 10). If $v'$ is a multiple of $x$, we have $Av' = 0$, $Ay = v'$, $Az = y$ (row 11).
        
        % If $v'$ is a multiple of $z$, we have $Ax = 0$, $Ay = x$, and $Av' = y$ (row 9). If $v'$ is a multiple of $x$, we have $Av' = 0$, $Ay = v'$, $Az = y$ (after relabeling $y, z$ to $x, y$ respectively, row 10). If $v'$ is in the span of $x, y$, then after scaling we have $Ax = 0$, $Av' = x$, and $Az = v'$ (after relabeling $z$ to $y$, row 11). If $v'$ is not in any of these categories, then $Ax = 0$, $Ay = x$, and $Av' = \lambda x + \mu y$; with change of basis, either $Av' = x$ (row 12) or $Av' = y$ (row 13). 
    \end{enumerate}
    \item If $A$ is two Jordan blocks, we casework on the generalized eigenvalues of $A$.
    \begin{enumerate}
        \item If the 2-dimensional Jordan block has nonzero generalized eigenvalue, then we can scale $A$ so that after a change of basis it is of the form $\begin{pmatrix} 1 & 1 & 0 \\ 0 & 1 & 0 \\ 0 & 0 & \lambda \end{pmatrix}$. If $v' = \alpha x + \beta y + \gamma z$, then $Av' = \lambda v' +  ((\lambda + 1)\alpha + \beta)x + (\lambda + 1)\beta y$. 
        
        We can scale $x, y$ together, translate $y$ by multiples of $x$, and scale $z$. If $\gamma \ne 0$, we can scale it to $1$ and then have two cases: if $\beta \ne 0$, then with basis $x, y, v'$, we have $Ax = x$, $Ay = y + x$, and $Av' = \lambda v' + (\lambda + 1)(x + y)$ (row 12); if $\beta = 0$, then with basis $x, y, v'$, we have $Ax = x$, $Ay = y + x$, and $Av' = v' + (\lambda +1)x$ (row 13).
        
        If $\gamma = 0$, we have two cases: if $\beta \ne 0$, then with basis $x, v', z$, we have $Ax = x$, $Av' = v' + x$, and $Az = \lambda z$ (after relabeling $z$ to $y$, row 14); if $\beta = 0$, then with basis $v', y, z$, we have $Av' = v'$, $Ay = v' + y$, and $Az = \lambda z$ (after relabeling $y, z$ to $x, y$ respectively, row 15).

        Note that if $\lambda = 1$, then $y$ can be shifted by $x$ or $z$, so we have only three cases: $v' = x$ (so $Av' = v'$, $Ay = v' + y$), $v' = y$ (so $Av' = v' + x$), or $v' = z$ (so $Av' = v'$ and $Ay = x + y$).
        \item If the 2-dimensional Jordan block has generalized eigenvalue 0, then $A$ is of the form $\begin{pmatrix} 0 & 1 & 0 \\ 0 & 0 & 0 \\ 0 & 0 & \lambda \end{pmatrix}$. If $\lambda \ne 0$, we can scale $x, y$, and $A$ so that $\lambda = 1$; otherwise $\lambda  = 0$. Suppose $v' = \alpha x + \beta y + \gamma z$.
        \begin{enumerate}
            \item If $\lambda = 1$, then as above, $x, y$ can be scaled together, $y$ can be translated by multiples of $x$, and $z$ can be scaled. Hence, we have four cases again. If $\gamma = 1$ and $\beta \ne 0$, then with basis $x, y, v'$, we have $Ax = 0$, $Ay = x$, and $Av' = v' + x + y$ (row 16); if $\beta = 0$, then with basis $x, y, v'$, we have $Ax = 0$, $Ay = x$, and $Av' = v' + x$ (row 17). If $\gamma = 0$ and $\beta \ne 0$, then with basis $x, v', z$, we have $Ax = 0$, $Av' = x$, and $Az = z$ (after relabeling $z$ to $y$, row 18); if $\beta = 0$, then with basis $v', y, z$, we have $Av' = 0$, $Ay = v'$, and $Az = z$ (after relabeling $y, z$ to $x, y$ respectively, row 19).
            \item If $\lambda = 0$, then $x, y$ can be scaled together, $y$ can be translated by multiples of $x$ or $z$, and $z$ can be scaled or translated by $x$. Hence, we have three cases: if $\beta \ne 0$, with basis $x, v', z$ we have $Ax = Az = 0$ and $Av' = x$ (row 20); if $\beta = 0$ but $\gamma \ne 0$, then with basis $x, y, v'$ we have $Ax = 0$, $Ay = x$, and $Av' = 0$ (row 21); if $\beta = \gamma = 0$, then with basis $v', y, z$ we have $Av' = 0$, $Ay = v'$, and $Az = 0$ (row 22).
        \end{enumerate}
    \end{enumerate}
\end{enumerate}
\end{proof}
\begin{proposition}
If $f = 0$, $L$ is abelian, and $b \ne 0$, then the isomorphism classes of such Lie algebra structures are: \begin{center}
    \begin{tabular}{| c | c | c | c | c | c |}
        & $[v, v]$ & $[v, v']$ & $[v, x]$ & $[v, y]$ & JNF of $A$ \\
        \hline
        4.1 & $x$ & $0$ & $0$ & $\mu y$ & $0$ \\ \hline
        4.2 & $v'$ & $0$ & $\lambda x$ & $\mu y$ & \multirow{4}{*}{$\diag(0, \lambda, \mu)$} \\ \cline{1-5}
        4.3 & $x$ & $0$ & $0$ & $y$ & \\ \cline{1-5}
        4.4 & $x$ & $v' + \alpha x + y$ & $0$ & $ \lambda y$ & \\ \cline{1-5}
        4.5 & $x$ & $v' + \alpha x$ & $0$ & $\lambda y$ & \\ \hline
        4.6 & $v'$ & $0$ & $v'$ & $x$ & \multirow{3}{*}{$\begin{pmatrix} 0 & 1 & 0 \\ 0 & 0 & 1 \\ 0 & 0 & 0 \end{pmatrix}$} \\ \cline{1-5}
        4.7 & $x$ & $\alpha x + y$ & $0$ & $\lambda x$ & \\ \cline{1-5}
        4.8 & $x$ & $\lambda x $ & $0$ & $v'$ & \\ \hline
        4.9 & $x$ & $y$ & $0$ & $0$ & \multirow{6}{*}{$\begin{pmatrix} 0 & 1 & 0 \\ 0 & 0 & 0 \\ 0 & 0 & 0 \end{pmatrix}$} \\ \cline{1-5}
        4.10 & $x$ & $0$ & $0$ & $v'$ & \\ \cline{1-5}
        4.11 & $v'$ & $0$ & $0$ & $x$ & \\ \cline{1-5}
        4.12 & $v'$ & $0$ & $v'$ & $0$ & \\ \cline{1-5}
        4.13 & $x$ & $x$ & $0$ & $0$ & \\ \cline{1-5}
        4.14 & $x$ & $0$ & $0$ & $x$ & \\ \hline
        4.15 & $v'$ & $0$ & $\lambda x$ & $\lambda (x + y)$ & \multirow{3}{*}{$\begin{pmatrix} 1 & 1 & 0 \\ 0 & 1 & 0 \\ 0 & 0 & 0 \end{pmatrix}$} \\ \cline{1-5}
        4.16 & $x$ & $v' + \lambda x + y$ & $0$ & $y$ & \\ \cline{1-5} 
        4.17 & $x$ & $v' + \lambda x$ & $0$ & $v' + y$ & \\ \hline
        4.18 & $v'$ & $0$ & $x$ & $\lambda y$ & \multirow{3}{*}{$\begin{pmatrix} 0 & 1 & 0 \\ 0 & 0 & 0 \\ 0 & 0 & 1 \end{pmatrix}$} \\ \cline{1-5}
        4.19 & $x$ & $v' + \lambda (x + y)$ & $0$ & $x$ & \\ \cline{1-5} 
        4.20 & $x$ & $\lambda x$ & $0$ & $y$ & \\ \hline
\end{tabular}
\end{center}
\end{proposition}
\begin{proof}
If $b \ne 0$, then $A$ has nontrivial kernel. 
\begin{enumerate}
    \item If $A = 0$, then either $b$ is a multiple of $v'$, so $[v, v] = v'$ (row 2 when $\lambda = \mu = 0$), or $b$ is not a multiple of $v'$, so $[v, v] = x$ (row 1).
    \item If $A$ is diagonalizable and nonzero, we casework on whether $b$ is a multiple of $v'$.
    \begin{enumerate}
        \item If $b$ is a multiple of $v'$, we can scale so that $b = v'$, but we can no longer scale $A$ so that one of the nonzero eigenvalues is $1$. So we have $Av' = 0$, $Ax = \lambda x$, and $Ay = \mu y$ (row 2).
        \item If $b$ is not a multiple of $v'$, then we can scale $A$ arbitrarily since $Ab = 0$ is not affected on scaling. Hence, WLOG we have that $A = \diag(0, 1, \lambda)$. If $Av' = 0$, then $\lambda = 0$ and we have $Av' = 0$, $Ax = 0$, and $Ay = y$ (row 3). If $Av' \ne 0$, then using similar logic as in the $b = 0$ case above, either $Ax = 0$, $Av' = v' + \alpha x + y$, and $Ay = \lambda y$ (row 4) or $Ax = 0$, $Av' = v' + \alpha x$, and $Ay = \lambda y$ (row 5). Note that if $\lambda = 1$, we only have the $Av' = v' + \alpha x$ case.
    \item If $A$ is nilpotent, we casework on its Jordan normal form.
    \begin{enumerate}
        \item If $A$ has one Jordan block, it is of the form $\begin{pmatrix} 0 & 1 & 0 \\ 0 & 0 & 1 \\ 0 & 0 & 0 \end{pmatrix}$ for basis $x, y, z$, and $b$ is a multiple of $x$. If $v'$ is also a multiple of $x$, then we can scale so that $[v, v] = v'$ and we have $Ay = v'$, $Az = y$ (after relabeling $y,z$ to $x, y$ respectively, row 6). Otherwise, we can scale so that $b = x$. Then the situation is the same as the single Jordan block/nilpotent case above, except that we can no longer scale $x$ and exclude the case where $v'$ is a multiple of $x$.
        Hence, either $Av' = \alpha x + y$, $Ax = 0$, $Ay = \lambda x$ (row 7) or $Av' = \lambda x$, $Ax = 0$, $Ay = v'$ (row 8).
        \item If $A$ has two Jordan blocks, it is of the form $\begin{pmatrix} 0 & 1 & 0 \\ 0 & 0 & 0 \\ 0 & 0 & 0 \end{pmatrix}$.
        \begin{enumerate}
            \item Suppose $b$ is in the 1-dimensional Jordan block. If $v'$ is not a multiple of $b$, the cases are the same as in the $b = 0$ case. Labeling $b = x$, either $Av' = y$, $Ay = 0$, and $Ax = 0$ (row 9) or $Av' = 0$, $Ay = v'$, and $Ax = 0$ (row 10). If $b$ is a multiple of $v'$, we can scale so that $b = v'$, then scale the other basis vectors so that $Ax = 0$, $Ay = x$, and $Av' = 0$ (row 11).
            \item Suppose $b$ is in the 2-dimensional Jordan block. If $b$ is a multiple of $v'$, then we can scale $A$ so that $b = v'$, in which case $Av' = 0$ and we can scale the other basis vectors so that $Ax = v'$, $Ay = 0$ (row 12).  If $b$ is not a multiple of $v'$, the cases are the same as in the $b = 0$ case -- either $Av' = x$, $Ay = 0$, $Ax = 0$ (row 13) or $Av' = 0$, $Ay = x$, $Ax = 0$ (row 14).
        \end{enumerate}
    \end{enumerate}
    \end{enumerate}
    \item If $A$ is not diagonalizable and not nilpotent, it still has nonzero kernel, so it either has a 1-dimensional or 2-dimensional Jordan block with generalized eigenvalue $0$.
    \begin{enumerate}
        \item In the former case, we can scale $A$ and change basis so that it is of the form $\begin{pmatrix} 1 & 1 & 0 \\ 0 & 1 & 0 \\ 0 & 0 & 0 \end{pmatrix}$ with basis $x, y, b$. If $b$ is a multiple of $v'$, then we can scale $A$ so $b = v'$, in which case we have $Ax = \lambda x$ and $Ay = \lambda (x + y)$ (row 15). Otherwise, scale $x$ so that $b = x$; the situation is similar to the case of two Jordan blocks where the 2-dimensional block has nonzero generalized eigenvalue, except that we can no longer scale $x$, so either $Av' = v' + y + \lambda x$, $Ax = 0$, and $Ay = y$ (row 16) or $Av' = v' + \lambda x$, $Ax = 0$, and $Ay = y + v'$ (row 17).
        \item Otherwise, we can scale $A$ to have nonzero eigenvalue $1$, and a change of basis has that $A$ is of the form $\begin{pmatrix} 0 & 1 & 0 \\ 0 & 0 & 0 \\ 0 & 0 & 1 \end{pmatrix}$ with basis $x, y, z$. We know $b$ is a multiple of $x$. If $b$ is a multiple of $v'$, we can scale $A$ so that $b = v'$; in that case, we have $Ay = x$ but $Az = \lambda z$ (after relabeling $z$ to $y$, row 18). Otherwise, we can scale $x$ so that $b = x$. Again, the situation is the same as in the $b = 0$ case except that we can no longer scale $x$ or $y$, but we can still translate $y$ by multiples of $x$. Hence, we have either $Av' = v' + \lambda (x + y)$, $Ax = 0$, $Ay = x$ (row 19) or $Av' = \lambda x$, $Az = z$, $Ax = 0$ (after relabeling $z$ to $y$, row 20).
    \end{enumerate}
\end{enumerate}
\end{proof}

\begin{proposition}
    If $f \ne 0$, then the isomorphism classes of such Lie algebras are:
\begin{center}
    \begin{tabular}{| c | c | c | c | c | c |}
        & $[v, v]$ & $[v, v']$ & $[v, x]$ & $[v, y]$ & $[x, y]$ \\
        \hline
        5.1 & $y$ & $0$ & $v + y$ & $0$ & $0$ \\ \hline
        5.2 & $y$ & $0$ & $v + v' + \lambda y$ & $0$ & $0$ \\ \hline
        5.3 & $y$ & $\lambda y + x$ & $v + \mu v' + \nu y$ & $v'$ & $0$ \\ \hline
        5.4 & $0$ & $0$ & $v + v' + \mu y$ & $0$ & $\lambda y$ \\ \hline
        5.5 & $0$ & $0$ & $v + y$ & $0$ & $\lambda y$ \\ \hline
        5.6 & $0$ & $0$ & $v + \mu v' + \nu y$ & $v'$ & $0$ \\ \hline
        5.7 & $0$ & $y$ & $v + \mu v' + \nu y$ & $\lambda v'$ & $0$ \\ \hline
        5.8 & $0$ & $0$ & $v + v' + \lambda y$ & $0$ & $y + v'$ \\ \hline
        5.9 & $0$ & $0$ & $v + y$ & $0$ & $y + v'$ \\ \hline
    \end{tabular}
\end{center}
\end{proposition}
\begin{proof}

If $f \ne 0$, then $L_0$ is a 2-dimensional Lie algebra with nontrivial center, so it must be abelian; say $L_0$ has basis $c, y$.

If $b \ne 0$, then $D$ is determined by $Dc = c$ and $Db = 0$. We have no conditions on $a$ other than $a(b) = 0$. $A_0: L_0 \to L_0$ satisfies $DA_0 + A_0 D = A_0$ and $A_0b = a(c)c$, so $A_0c$ is a multiple of $b$. Then
\begin{enumerate}
    \item If $a = 0$ and $A_0c = 0$, then we may scale $A$ so that either $q = b$ or $q = c + \lambda b$. Hence, we have $[x, v'] = v'$, $[x, y] = 0$, $[v, v] = y$,  $[v, v'] = 0$, $[v, y] = 0$, and $[v, x] = v' + \lambda y + v$ (row 2) or $v + y$ (row 1).
    \item If $a \ne 0$, then we can scale so that $a(c) = 1$. Hence $A_0 b = c$, $A_0 c = \lambda b$, and $q$ is unrestricted. We have the brackets $[x, v'] = v'$, $[x, y] = 0$, $[v, v] = y$, $[v, x] = \mu v' + \nu y + v$, $[v, y] = v'$, $[v, v'] = \lambda y + x$ (row 3).
\end{enumerate}

If $b = 0$, we always have $0 = A_0 b = a(c) c$, so $a(c) = 0$. Since $a \wedge D = 0$ and $Dc \ne 0$, we require $a(y) = 0$, so in fact $a = 0$. Either $D$ is diagonalizable or not.
\begin{enumerate}
    \item $D$ is diagonalizable, so $Dc = c$ and $Dy = \lambda y$. That $[D, A_0] = A_0$ implies that $A_0c = \alpha y$, $A_0 y = \beta c$, and $(\lambda + 1)A_0 = A_0$. So either $\lambda = 0$, or $A_0 = 0$. 
    \begin{enumerate}
        \item If $A_0 = 0$, we may scale $A$ so that either $q = b$ or $q = c + \mu b$. Meanwhile, $\lambda$ can be any constant. In this case, the brackets are $[x, v'] = v'$, $[x, y] = \lambda y$, $[v, v'] = 0$, $[v, y] = 0$, $[v, x] = v + v' + \mu y$ (row 4) or $[v, x] = v + y$ (row 5).
        \item If $A_0 \ne 0$, then $\lambda = 0$. Then we have two subcases based on whether $\alpha = 0$ or not:
        \begin{enumerate}
            \item If $\alpha = 0$, we may scale $A$ so that $\beta = 1$. Hence the brackets are $[x, v'] = v'$, $[x, y] = 0$, $[v, x] = v + \mu v' + \nu y$, $[v, v'] = 0$, $[v, y] = v'$ (row 6).
            \item If $\alpha \ne 0$, we may scale $A$ so that $\alpha = 1$, while $\beta$ can be any constant. Hence the brackets here are $[x, v'] =  v'$, $[x, y] = 0$, $[v, x] = v + \mu v' + \nu y$, $[v, v'] = y$, $[v, y] = \beta v'$ (row 7).
        \end{enumerate}
    \end{enumerate}
    \item $D$ is not diagonalizable, so $Dc = c$ and $Dy = c + y$. Then we can check that $[D, A_0] = A_0$ implies that $A_0 = 0$. So in this case $[x, v'] = v'$, $[x, y] = y + v'$, $[v, v'] = 0$, $[v, y] = 0$, and$[v, x] = v + v' + \lambda y$ (row 8) or $[v, x] = v + y$ (row 9).
\end{enumerate}
\end{proof}

\subsection{Classification of Lie algebra structures on $2P$}

In this section, we classify Lie algebra structures on $\mfrak{g} = 2P$. Let $L = \ker d$ and choose a splitting $\mfrak{g} = V \oplus L$ as vector spaces as above. We must pick $L$, an $L$-module structure on $V$, a derivation $A: L \to \Hom(V, L)$, a element $B \in V$ (corresponding to $B(x \otimes y$), and elements $C(x \otimes y), C(x \otimes x), C(y \otimes y) \in L$ satisfying the above conditions.

If $L$ is abelian, then $f = 0$ and we must choose a 2-dimensional Lie algebra $\mfrak{g}_0$, a 2-dimensional representation $L$ of $\mfrak{g}_0$, an injective 1-cocycle $d: \mfrak{g}_0 \to L$, and a symmetric map $\mathbf{C}: S^2(\mfrak{g}_0) \to L$ up to coboundaries such that $p(C) \in L^{\mfrak{g}_0}$, where $C \in H^2_e(\mfrak{g}_0, L)$ is the extended cohomology class of $[\mathbf{C}]$). Note that $\mathbf{C}$ is automatically a 2-cocycle. Let $\mathbf{C}_{zw} = C(z \otimes w)$ for $z, w \in \mfrak{g}_0$.

\begin{proposition}
    If $L$ and $\mfrak{g}_0$ are abelian, we have $x', y'$ central and the following cases for the other brackets:
    \begin{center}
        \begin{tabular}{| c | c | c | c |}
        & $[x, x]$ & $[y, y]$ & $[x, y]$ \\ \hline
        1 & $0$ & $0$ & $0$ \\ \hline
        2 & $0$ & $0$ & $x'$ \\ \hline
        3 & $y'$ & $0$ & $\lambda x', \lambda \ne 0$ \\ \hline
        4 & $y'$ & $0$ & $y'$ \\ \hline
        5 & $x' + \lambda y'$ & $0$ & $x'$ \\ \hline
        6 & $x'$ & $0$ & $\lambda y'$ \\ \hline
        7 & $y'$ & $\lambda x', \lambda \ne 0$ & $\mu x'$\\ \hline
        8 & $x'$ & $y'$ & $\lambda x' + \mu y'$ \\ \hline
        \end{tabular}
    \end{center}
\end{proposition}
\begin{proof}
    If $\mfrak{g}_0$ is abelian, then $d$ is any injective map and $C$ is any symmetric 2-form on $\mfrak{g}_0$ taking values in $L$, so we need to choose elements $C_{xy}, C_{xx}, C_{yy} \in L$. Then $[x, y] = C_{xy}$, $[x, x] = C_{xx}$, $[y, y] = C_{yy}$, and all other brackets are zero. Therefore, we have the following cases:
    \begin{enumerate}
        \item If $[x, x] = [y, y] = 0$, then either $[x, y] = 0$ and $\mfrak{g}$ is abelian (row 1) or $[x, y] = v'$ where WLOG $v$ is linearly independent of $y$. Then with basis $v, y$, we have $[v, y] = v'$. So we can say $[x, y] = x'$ (row 2).
        \item If $[x, x]$ and $[y, y]$ are linearly dependent but not both zero, WLOG say $[x, x] \ne 0$ and $[y, y] = \gamma^2 [x, x]$. Then $[y + \gamma x, y + \gamma x] = 0$. So we can say $[y, y] = 0$. Now either $[x, x]$ is a multiple of $y'$, so by scaling $y$ we have $[x, x] = y'$ , or $[x, x] = v'$, $v$ linearly independent of $y$, and so $[v, v] = \lambda v'$ for some $\lambda \ne 0$, and by rescaling $v$ we get $[v, v] = v'$ .
        \begin{enumerate}
            \item Suppose $[x, x] = y'$. Then if $[x, y] = v'$, $v$ linearly independent of $y$, we have $[v, y] = \lambda v'$ for some $\lambda \ne 0$ and $[v, v] = \lambda^2 y'$. By rescaling $v$, we can have $[v, v] = y'$, but we now cannot scale $y$, so $[v, y] = \lambda v'$ still (row 3). If instead $[x, y] = \lambda y'$ instead, then we can scale $x \mapsto \lambda^{-1}x$ and $y \mapsto \lambda^{-2}y'$ to get $[x, y] = y'$ (row 4).
            \item Suppose $[x, x] = x'$. Then if $[x, y] = v'$, $v$ linearly independent of $y$, we have $[v, y] = \lambda v'$ and $[v, v] = \lambda v' + \mu y'$ for $\lambda \ne 0, \mu \in k$. By scaling $v \mapsto \lambda^{-2}v$ and $y \mapsto \lambda^{-1}y$, we can get $[v, v] = v' + \mu y'$ and $[v, y] = v'$ (row 5). If instead $[x, y] = \lambda y'$, we can't do any more scaling (row 6).
        \end{enumerate}
        \item If $[x, x]$ and $[y, y]$ are linearly independent, suppose first that $[y, y]$ is linearly independent of $y'$. WLOG we can set $[x, x] = \lambda y'$ and $[y, y] = x' + \mu y'$. By scaling $x$ and translating $y$, we can get $[x, x] = y'$ and $[y, y] = \nu x'$ for $\nu \ne 0$. Now if $[x, y] = \alpha x' + \beta y'$, we cannot scale $x$ or $y$ anymore, but by picking $\mu$ such that $\mu(\alpha + \mu \lambda) = \beta$, we have $[x + \mu y, y] = (\alpha + \mu \lambda)(x' + \mu y')$, so we can get $[x, y'] = \alpha x'$ (row 7).
        
        If instead $[y, y]$ is a multiple of $y'$, we know $[x, x]$ is linearly independent of $y'$. Hence, we can scale $x, y$ so $[y, y] = y'$ and $[x, x] = x' + \mu y'$. Notice that $$[x + \lambda y, x + \lambda y] = x' + \mu y' + \lambda^2 y',$$ so by choosing $\lambda = \mu + \lambda^2$ and changing $x \mapsto x + \lambda y$, we get $[y, y] = y'$ and $[x, x] = x'$. Now $[x, y]$ can be anything (row 8).
    \end{enumerate}
\end{proof}

\begin{proposition}
    If $L$ is abelian and $\mfrak{g}_0$ is not abelian, we have the following cases (we list the $x$- and $y$-actions on $L$ as matrices in a basis $v, w$, and unless otherwise specified $x' = v$, $y' = w$.
    \begin{center}
    \begin{tabular}{| c | c | c | c | c | c | c | }
    & $[x, x]$ & $[y, y]$ & $[x, y]$ & $\rho(x)$ & $\rho(y)$ & add'l conditions \\
        \hline
        1 & $0$ & $0$ & $x$ & $\begin{pmatrix} 0 & 1 \\ 1 & 0 \end{pmatrix}$ & $\begin{pmatrix} \alpha & 0 \\ 0 & \alpha + 1 \end{pmatrix}$ & \makecell{ $x' = \lambda v + \mu w$ \\ $y' = \alpha \mu v + (\alpha + 1)\lambda w$} \\ \hline
     % 1 & $0$ & $0$ & $x$ & \multirow{2}{*}[-0.5em]{$\begin{pmatrix} 1 & 1 \\ 0 & 1 \end{pmatrix}$} & \multirow{2}{*}[-0.5em]{$\begin{pmatrix} 0 & b \\ 1 & 1 \end{pmatrix}$} & \makecell{ $x' = v$ \\ $y' = w$} \\ \cline{1-1} \cline{7-7}
        % 2 & $0$ & $0$ & $x$ &  &  & \makecell{$x' = cv + w$ \\ $y' = bv + cw$} \\ \hline
        2 & $0$ & $0$ & $x$ & $\begin{pmatrix} 0 & 0 \\ \alpha + 1 & 0 \end{pmatrix}$ & $\begin{pmatrix} \alpha & \lambda \\ 0 & \alpha + 1 \end{pmatrix}$ & $\alpha \ne 0$ \\ \hline
        3 & $\omega v$ & $\xi v$ & $x + \nu w$ & $\begin{pmatrix} 0 & 1 \\ 0 & 0 \end{pmatrix}$ & $\begin{pmatrix} 0 & 0 \\ 0 & 1 \end{pmatrix}$ & \makecell{ $x' = av + bw$ \\ $y' = cv + aw$ \\ $a^2 \ne bc$} \\ \hline
        % 4 & \multirow{2}{*}[-0.5em]{$\omega v$} & \multirow{2}{*}[-0.5em]{$\xi v$} & \multirow{2}{*}[-0.5em]{$x + \nu w$} & \multirow{2}{*}[-0.5em]{$\begin{pmatrix} 0 & 1 \\ 0 & 0 \end{pmatrix}$} & \multirow{2}{*}[-0.5em]{$\begin{pmatrix} 0 & 0 \\ 0 & 1 \end{pmatrix}$} & \makecell{$x' = w$ \\ $y' = cv$} \\ \cline{1-1} \cline{7-7}
        % 5 & & & & & & \makecell{$x' = v + bw$ \\ $y' = cv + w$} \\ \hline
        4 & $0$ & $0$ & $x + \mu y'$ & $0$ & $\begin{pmatrix} 1 & 1 \\ 0 & 1 \end{pmatrix}$ & \\ \hline 
        5 & $0$ & $0$ & $x + \mu y'$ & \multirow{2}{*}{$0$} & \multirow{2}{*}{$1$} & \\ \cline{1-4}
        6 & $0$ & $0$ & $x + x' + \mu y'$ &  &  & \\ \hline
        7 & $\lambda y'$ & $\mu y'$ & $x$ & \multirow{2}{*}{$0$} & \multirow{2}{*}{$\diag(1, 0)$} & \\ \cline{1-4}
        8 & $\lambda y'$ & $\mu y'$ & $x + x'$ & & & \\ \hline
        9 & $0$ & $0$ & $x$ & \multirow{2}{*}{$0$} & \multirow{2}{*}{$\diag(1, \lambda)$} & \multirow{2}{*}{$\lambda \ne 0, 1$} \\ \cline{1-4}
        10 & $0$ & $0$ & $x + x'$ & & & \\ \hline
    \end{tabular}
    \end{center}
    Row 1 is up to simultaneous scaling of $\lambda, \mu$, while row 3 is up to simultaneous scaling of all the parameters.
\end{proposition}
\begin{proof}
    If $\mfrak{g}_0$ is nonabelian, it has nonzero bracket $[x, y] = x$, and $L$ is a two-dimensional representation of $\mfrak{g}_0$. Then the condition for $d$ to be a 1-cocycle is that
    \begin{align*}
        d[x, y] &= \rho(y)dx + \rho(x)dy \\
        \implies (\rho(y) + 1)dx &= \rho(x)dy
    \end{align*}
    and the 2-coboundaries are maps $V \otimes V \to L$ with
    \begin{align*}
        x \otimes x &\mapsto 0 \\
        y \otimes y &\mapsto 0 \\
        x \otimes y &\mapsto (\rho(y) + 1)\varphi(x) + \rho(x)\varphi(y)
    \end{align*}
    for $\varphi: V \to L$. In particular, if either of $\rho(y) + 1$ or $\rho(x)$ is invertible, WLOG we can set $\mathbf{C} = 0$.
    
    We know $\mfrak{g}_0$ has two families of irreducible representations: the 1-dimensional representations where $y \mapsto \lambda \in k$ and $x \mapsto 0$, and the 2-dimensional representations where $y \mapsto \begin{pmatrix} \alpha & 0 \\ 0 & \alpha + 1 \end{pmatrix}$ and $x \mapsto \begin{pmatrix} 0 & b \\ b & 0 \end{pmatrix}$, $b \ne 0$, $\alpha, b \in k$. Let $v, w$ be a basis of $L$ such that $\rho(x), \rho(y)$ act as these matrices.
    
    Then, if $L$ is 2-dimensional irreducible, we can first scale so that $b = 1$ since $b \ne 0$. Also, $L^{\mfrak{g}_0} = 0$ so $p(C) = 0$. Since $\rho(x)$ is invertible, we can set $\mathbf{C} = 0$. Therefore, $C = 0$, and we always have $[x, y] = x$, $[x, x] = 0$, $[y, y] = 0$. It remains to determine $d$. If $dx = \lambda v + \mu w$, then the 1-cocycle condition implies $dy = \alpha\mu v + (\alpha + 1)\lambda w$. By simultaneously scaling $v, w$, we can simultaneously scale $\lambda, \mu$. (row 1)
    
    Otherwise, $L$ is an extension of two 1-dimensional representations, say $\rho(x) = \begin{pmatrix} 0 & \lambda \\ 0 & 0 \end{pmatrix}$ and $\rho(y) = \begin{pmatrix} \alpha & \mu \\ 0 & \beta \end{pmatrix}$. 
    \begin{enumerate}
    
    \item If $\rho(x) \ne 0$, scale $x$ so $\lambda = 1$, translate $y$ so $\mu = 0$, and we need $\beta = \alpha + 1$. Therefore, $xv = 0$, $xw = v$, $yv = \alpha v$, and $yw = (\alpha + 1)w$. Writing $dx = av + bw$ and $dy = cv + ew$, the 1-cocycle condition imposes conditions on $a, b, c, e$ based on the value of $\alpha$. If $\alpha \ne 0$, $L^{\mfrak{g}_0} = 0$ while the image of the 2-coboundaries is $L$; otherwise, $L^{\mfrak{g}_0}$ and the image of the 2-coboundaries are both $\text{span}(v)$. 

    If $\alpha \ne 0$, we can set $C = 0$. In order for $d$ to be a 1-cocycle, we need $b = 0$, so $dx = av$ with $a \ne 0$, and then $dy = cv + (\alpha + 1)aw$. So the brackets in this case are $[x', x] = 0$, $[x', y] = \alpha x'$, $[y', x] = (\alpha + 1)x'$, $[y', y] = a^{-1}cx' + (\alpha + 1)y'$. (row 2)

    If $\alpha = 0$, $\mathbf{C}_{xx}, \mathbf{C}_{yy}$ are multiples of $v$ and $\mathbf{C}_{xy}$ can be translated to a multiple of $w$. The 1-cocycle condition implies $a = e$, and $d$ is injective, so $a^2 \ne bc$. Hence, the brackets are
    \begin{align*}
        [x', x] &= bv \\
        [x', y] &= bw \\
        [y', x] &= av \\
        [y', y] &= aw \\
        [x, x] &= pv \\
        [y, y] &= qv \\
        [x, y] &= x + rw
    \end{align*}
    up to simultaneous scaling, since $v,w$ can be simultaneously rescaled. (row 3)

    \item If $\rho(x) = 0$, we casework on the Jordan normal form of $\rho(y)$. Note that for $d$ to be a bijective 1-cocycle, we need $(\rho(y) + 1)dx = 0$, so $\rho(y) + 1$ must have nonzero kernel. Therefore, at least one (generalized) eigenvalue of $\rho(y)$ must be $1$.

    \begin{enumerate}

    \item If $\rho(y)$ is not diagonalizable, then $\rho(y) = \begin{pmatrix} 1 & 1 \\ 0 & 1 \end{pmatrix}$. Then $dx$ is a multiple of $e_0$ while $dy$ can be anything; we can translate $y$ by $x$ so that $dy$ is a multiple of $e_1$ and scale $x$ so that $dx = \lambda e_0$, $dy = \lambda e_1$. Since $\rho(y)$ is invertible, $\mathbf{C}_{xx} = \mathbf{C}_{yy} = 0$, and $\rho(y) + 1$ has image $\text{span}(e_0)$, so $\mathbf{C}_{xy}$ can be translated to a multiple of $e_1$. Then $[x', x] = [y', x] = 0$, $[x', y] = x'$, and $[y', y] = y' + x'$. Meanwhile $[x, x] = [y, y] = 0$, and $[x, y] = x + \mu y'$. (row 4)

    \item If $\rho(y) = 1$, then $\mathbf{C}_{xx} = \mathbf{C}_{yy} = 0$, while $\mathbf{C}_{xy}$ can be anything. So $[x', x] = [y', x] = 0$, $[x', y] = x'$, $[y', y] = y'$, $[x, x] = [y, y] = 0$, and $[x, y] = x + \lambda x' + \mu y'$. By scaling $x$, $\lambda = 0, 1$. (rows 5 and 6)

    \item If $\rho(y) = \diag(1, 0)$, then $\mathbf{C}_{xx}, \mathbf{C}_{yy} \in \text{span}(e_1)$, while $\mathbf{C}_{xy}$ is a multiple of $e_0$. Then $dx$ is a multiple of $e_0$ while $dy$ can be anything; again, translating $y$ by $x$ we can make it so that $dy$ is a multiple of $e_1$. Scaling, $x$ we can make it so that either $\mathbf{C}_{xy} = dx$ or $\mathbf{C}_{xy} = 0$. Therefore, the nonzero brackets are $[x, x] = \lambda y'$, $[y, y] = \mu y'$, $[x', y] = x'$, and $[x, y] = x + x'$ or $[x, y] = x$. (rows 7 and 8)
    
    \item If $\rho(y) = \diag(1, \lambda)$ for $\lambda \ne 1, 0$, then the situation is the same as above except that $\mathbf{C}_{xx} = \mathbf{C}_{yy} = 0$. Therefore, the nonzero brackets are $[y', y] = \lambda y'$, $[y', x] = x'$, and $[x, y] = x + x'$ or $[x, y] = x$. (row 9 and 10)
    \end{enumerate}
\end{enumerate}
\end{proof}
\begin{remark}
    Row 3 with $\xi = b = c = 0$, $\omega = a = 1$ is the Lie algebra structure on $\mfrak{gl}(P)$, as with basis $x, x', y, y'$ described in Section \ref{glp_irreps}, we have $[x, y] = x + y'$, $[x, x] = x'$, $[y, y] = 0$, $[x, y'] = x'$, $[y, y'] = y'$, and $x'$ central.
\end{remark}

\begin{proposition}
    If $L$ is not abelian, we have the following two cases, where $[x', y'] = x'$:
    \begin{center}
    \begin{tabular}{| c | c | c | c | c | c | c | }
    $[x, x]$ & $[y, y]$ & $[x, y]$ & $[x', x]$ & $[x', y]$ & $[y', x]$ & $[y', y]$ \\ \hline
    \multirow{2}{*}{$0$} & \multirow{2}{*}{$y'$} & \multirow{2}{*}{$\mu x'$} & $x'$ & $x + y'$ & $x + y'$ & \multirow{2}{*}{$\nu x'$} \\ \cline{4-6} 
     &  & & $0$ & $x$ & $x$ & \\ \hline 
    \end{tabular}
    \end{center}
\end{proposition}
\begin{proof}
    Say that $L$ has nonzero bracket $[z, w] = z$. Then we must determine $f$, $A$, $B$, and $C$ as described in \ref{lie_m_np_classification}. Setting $x' = z$ and $y' = w$, we have $f(x')y = f(y')x = x$. Then 
    \begin{align*}
        f(C(x \otimes x)) &= f(x')f(x') = 0 \implies C(x \otimes x) = 0 \\
        f(C(y \otimes y)) &= f(y')f(y') = f(y') \implies C(y \otimes y) = y'.
    \end{align*}
    We have
    $$A(C(x \otimes x)) = 0 = A(x')f(x') + (\ad x')A(x')$$ implies that $[x', A(x')x] = 0$, so $A(x')x = \lambda x'$, and $A(x')x = [x', A(x')y]$, so $A(x')y = \lambda y'$. By scaling $x$ and $x'$, we can make it so that $\lambda = 1$ or $0$.
    Also,
    $$A(C(y \otimes y)) = A(y') = A(y')f(y') + (\ad y')A(y').$$
    Hence, $[y', A(y')x] = 0$ and $A(y')x = \mu y'$, and $A(y')y = [y', A(y')y]$ so $A(y')y = \nu x'$.
    Substituting $x', y', x$ into the Jacobi identity (Eq. \ref{a_derivation}, see Section \ref{gl_mnp_framework}), we get that
    $[A(y')x, x'] = \lambda x$, so $\mu = \lambda$.
    Substituting $x', y', y$ into Eq. \ref{a_derivation}, we get that $[A(y')y, x'] = 0$, which is already true. So we conclude that $A(x') = \lambda I, A(y') = \begin{pmatrix} 0 & \nu \\ \lambda & 0 \end{pmatrix}$. This implies that $B(x \otimes y) = 0$.
    
    The condition on $[B, z]$ gives no additional information, while
    the condition on $[C, z]$ implies that $[C(x \otimes y), x'] = 0$ and $[C(x \otimes y), y'] = C(x \otimes y)$. Therefore, $C(x \otimes y) = \mu x'$. Hence, the nonzero brackets are:
    \begin{itemize}
        \item $[x, y] = \mu x'$
        \item $[y, y] = y'$
        \item $[x', y] = [x, y'] = x + \lambda y'$
        \item $[x', x] = \lambda x'$
        \item $[y', y] = \nu x'$
        \item $[x', y'] = x'$
    \end{itemize}
    where $\lambda = 0, 1$ and $\mu, \nu \in k$.
\end{proof}

\section{A propsed analog of the $p$-center for Lie algebras in $\Ver_4^+$}\label{pcenter}
In this section, we begin investigating the center of the universal enveloping algebra  $U(\mfrak{g})$ of a Lie algebra $\mfrak{g}$ by conjecturing an analog of the $p$-center for $\mfrak{g}$, but in this case, we will have a ``$4$-center'' rather than a $2$-center. We first need a notion of a restricted Lie algebra.
\begin{definition}
    Let $\mfrak{g}$ be a Lie algebra in $\Ver_4^+$ with $\mfrak{g} \subset \mfrak{gl}(m \cdot \1 + nP)$. Recall that there is a multiplication map $\mfrak{gl}(m \cdot \1 + nP) \times \mfrak{gl}(m \cdot \1 + nP) \to \mfrak{gl}(m \cdot \1 + nP)$ given by evaluation. We say that $\mfrak{g}$ is restricted if for all $x \in \mfrak{g}$, we have $x^{[2]} \in \mfrak{g}$, where $x^{[2]}$ is the square of $x$ taken in $\mfrak{gl}(m \cdot \1 + nP)$.
\end{definition}

\begin{remark}
    The classical definition of a restricted Lie algebra, that there exists a map $\mfrak{g} \to \mfrak{g}, x \mapsto x^{[2]}$ such that the map $\xi: \mfrak{g} \to U(\mfrak{g})$ defined by $x \mapsto x^2 - x^{[2]}$, is semilinear and takes values in $Z(\mfrak{g})$, is no longer valid, since it is no longer the case that $x^2 - x^{[2]}$ is central. However, we can still define the notion of a restricted algebra for Lie subalgebras of $\mfrak{gl}(X)$.
\end{remark}

Let $\mfrak{g}$ be a restricted Lie algebra. Denote by $X^2$ the square of $X$ in $U(\mfrak{g})$. For $X, Y \in \mfrak{g}$,
\begin{equation*}
    [X, [X, Y]] = X^{[2]}Y+YX^{[2]} + Y'(X^{[2]})'+X'YX'+YX'X' = [X^{[2]}, Y] + [X', YX'].
\end{equation*}

Therefore, if $X \in \ker d$, $X^2 - X^{[2]}$ is central, but when $X' \ne 0$, we instead find that a degree 4 polynomial in $X$ will be central given some conditions on $X, X'$. That is:
\begin{itemize}
    \item If $X^{[2]} = {X'}^{[2]} = 0$, then
    \begin{equation*}
    \ad(X)^4(Y) = X'(X'YX')X' = 0
\end{equation*}
so $X^4$ is central.
\item If $X^{[2]} = 0$ and ${X'}^{[2]} = X'$, then 
\begin{equation*}
    \ad(X)^4(Y) = X'YX' + X'YX' + X'YX' + YX' = X'YX' + YX'.
\end{equation*}
so $X^4 - X^2$ is central.
\item If $X^{[2]} = X$ and ${X'}^{[2]} = 0$, $[X, X'] = X'$ or $0$, then \begin{equation*}
    \ad(X)^4(Y) = [X, [X, Y]] + XX'YX' + X'YX'X+ X'XYX' + X'YXX' = 0,
\end{equation*}
so $X^4$ is central.
\end{itemize}

Let $x_1, \dots, x_m$ be a basis for $m \cdot \1$ and $y_1, \dots, y_n, y'_1, \dots, y'_n$ be a basis for $nP$, with $y'_i = dy_i$, with corresponding dual basis $x_i^*, y_i^*, {y'_i}^*$. Then
\begin{itemize}
    \item $x_i \otimes x_j^*$, $dy_i \otimes x_j^*$, and $x_i \otimes y_j^*$ are in the kernel of $d$
    \item $x_i \otimes {y'_j}^*$, $y_i \otimes x_j^*$, $y_i \otimes y_j^*$ when $i \ne j$, and $dy_i \otimes dy_j^*$ when $i \ne j$ satisfy $X^{[2]} = {X'}^{[2]} = 0$
    \item $y_i \otimes {y'_i}^*$ satisfies $X^{[2]} = 0, {X'}^{[2]} = X'$
    \item $y_i \otimes y_i^*$ satisfies $X^2 = X, {X'}^{[2]} = 0, [X, X'] = X'$.
\end{itemize}

\begin{proposition}
    The following elements are central in $U(\mfrak{gl}(m \cdot \1 + nP))$:
    \begin{itemize}
        \item $(x_i \otimes x_j^*)^2$, $(y'_i \otimes x_j^*)^2$, $(x_i \otimes y_j^*)^2$ for all $i, j$,
        \item $(x_i \otimes y'_j)^4$, $(y_i \otimes x_j^*)^4$ for all $i, j$, $(y_i \otimes y_j^*)^4$, $(y'_i \otimes {y'_j}^*)^4$ for $i \ne j$,
        \item $(y_i \otimes {y'_i}^*)^4 - (y_i \otimes {y'_i}^*)^2$ ,
        \item $(y_i \otimes y_i^*)^4$.
    \end{itemize}
    We call this the ``$4$-center'' of $U(\mfrak{gl}(m \cdot \1 + nP))$.
\end{proposition}

\begin{conjecture}
  In $\mfrak{gl}(m \cdot \1 + nP)$, and hence in any restricted subalgebra, if $X \notin \ker d$, some degree 4 polynomial $X^4 + {X^{[2]}}^2 + F(X, X')$ is
  central, where $F$ is a noncommutative polynomial function of $X, X^{[2]}$ and $X'$ of degree at most 4.
\end{conjecture}

\section{Quadratic Casimir elements of $\mfrak{gl}(m \cdot \1 + P)$}

In this section, we find quadratic Casimir elements of $U(\mfrak{gl}(m \cdot \1 + nP))$, giving central elements of $U(\mfrak{gl}(m \cdot \1 + nP))$ not in the ``$4$-center''.

Recall that given a nondegenerate bilinear form $B: \mfrak{g} \otimes \mfrak{g} \to k$ that is invariant under the adjoint action, i.e.
\begin{equation*}
    B(\ad_X Y, Z) + B(Y, \ad_X Z) = 0,
\end{equation*}
we get an isomorphism $\mfrak{g}^* \cong \mfrak{g}$. The Casimir element for $B$ is the image of
\begin{equation*}
    \1 \xrightarrow{coev} \mfrak{g} \otimes \mfrak{g}^* \xrightarrow{B} \mfrak{g} \otimes \mfrak{g} \to U(\mfrak{g})
\end{equation*}
and will be central in $U(\mfrak{g})$.

The adjoint-invariant condition on $B$ is equivalent to the isomorphism $\varphi: \mfrak{g} \cong \mfrak{g}^*$ satisfying 
\begin{equation*}
    \varphi_{[x, y]}(-) = \varphi_y([x, -]),
\end{equation*}
i.e. $\varphi$ is a map of $\mfrak{g}$-modules.

\begin{remark}
    One can check that the Killing form $B(x,y) = \Tr (\ad x \circ \ad y)$ is zero on $\mfrak{gl}(P)$.
\end{remark}

Again, let $x_1, \dots, x_m$ be a basis for $m \1$ and $y_1, \dots, y_n$, $y'_1, \dots, y'_n$ be a basis for $nP$, with dual basis $x_i^*, y_i^*, {y'_i}^*$. Note that $d({y'_i}^*) = y_i^*$. Let $S := \{x_1, \dots, x_m, y_1, \dots, y_n, y'_1, \dots, y'_n\}$; then a basis for $\mfrak{gl}(m \cdot \1 + nP) = (m \cdot \1+nP) \otimes (m \cdot \1+nP)^*$ is $z \otimes w^*, z, w \in S$. 

\begin{theorem}
    Adjoint-invariant bilinear forms on $\mfrak{g} = \mfrak{gl}(m \cdot \1 + nP)$ are classified by two parameters, $\lambda$ and $\mu$.
    When $m$ is odd, a bilinear form $B_{\lambda, \mu}$ on $\mfrak{gl}(m \cdot \1 + nP)$ is defined by
    \begin{align*}
        B(x, y) &= \mu \Tr(xy), x, y \in \mfrak{sl}(m \cdot \1 + nP) \\
        B(I, I) &= \lambda \\
        B(I, x) &= 0, x \in \mfrak{sl}(m \cdot \1 + nP)
    \end{align*}
    where $I = \sum_{z \in S} z \otimes z^*$ is the identity matrix.
    When $m$ is even, a bilinear form $B_{\lambda, \mu}$ on $\mfrak{gl}(m \cdot \1 + nP)$ is defined by 
    \begin{align*}
        B(x, y) &= \mu \Tr(xy), x, y \in \mfrak{sl}(m \cdot \1 + nP) \\
        B(h_{y_n}, h_{y_n}) &= \lambda \\
        B(h_{y_n}, I) &= \mu \\
        B(h_{y_n}, x) &= 0, x = z \otimes w^*, z \ne w \text{ or } z \otimes z^* + y'_n \otimes {y'_n}^*
    \end{align*}
    where $h_{y_n} := y_n \otimes y_n^*$.
    In both cases, $B$ is symmetric. It is nondegenerate when
    \begin{itemize}
        \item $\mu \ne 0$ if $m$ is even;
        \item $\lambda \ne 0$ if $m$ is odd.
    \end{itemize}
    The Casimir element associated to $B_{\lambda, \mu}$ is 
\begin{equation*}
    \frac{\lambda}{\mu^2}\sum_{i = 1}^m ((x_i \otimes x_i^*)^2 + x_i \otimes x_i^*) + \frac{\lambda + \mu}{\mu^2}\sum_{z \in S} (z \otimes z^*).
\end{equation*}
when $m$ is even and
\begin{equation*}
    \frac{1}{\lambda} \sum_{i = 1}^m ((x_i \otimes x_i^*)^2 + x_i \otimes x_i^*) + \frac{1}{\lambda} \sum_{z \in S} (z \otimes z^*).
\end{equation*}
when $m$ is odd.
\end{theorem}

We first show the theorem for $\mfrak{gl}(P)$.
\begin{proposition}
    $\mfrak{gl}(P)$ has a two-dimensional space of invariant bilinear forms.
\end{proposition}
\begin{proof}
    Let $I := y \otimes y^* + y' \otimes {y'}^*$, $y \otimes {y'}^*$, $y' \otimes y^*$, and $y \otimes y^*$ be a basis for $\mfrak{gl}(P)$. Then writing out all the brackets, we check that $B$ is symmetric and that the only nonzero pairings are
    \begin{itemize}
        \item $B(y \otimes y^*, y \otimes y^*)$
        \item $B(I, y \otimes y^*) = B(y \otimes {y'}^*, y' \otimes y^*)$
    \end{itemize}
    and these are independent of each other. So $B$ is parametrized by 
    \begin{align*}
        \lambda &:= B(y \otimes y^*, y \otimes y^*), \\
        \mu &:= B(I, y \otimes y^*).
    \end{align*}
\end{proof}

The rest of this section is devoted to proving this theorem for other $m, n$.

\begin{definition}
    Let $\mfrak{sl}(X)$ be the kernel of the evaluation map $\mfrak{gl}(X) \to \1$; in other words, the Lie subalgebra of traceless matrices.
\end{definition}
\begin{proposition}
 If $X = m \cdot \1 + nP$ for odd $m$, $\mfrak{gl}(X) = \1 \oplus \mfrak{sl}(X)$, where $\1$ is the central subalgebra that is the image of $\coev: \1 \to X \otimes X^*$. If $m$ is even, $\mfrak{gl}(X)$ is a nontrivial extension of $\1$ by $\mfrak{sl}(X)$.
\end{proposition}

We first show that there is only one adjoint-invariant bilinear form up to scalars on $\mfrak{sl}(m \cdot \1 + nP)$, then extend it to $\mfrak{gl}(m \cdot \1 + nP)$.

\begin{proposition}
    If $m$ is odd, $\mfrak{sl}(m \cdot \1 + nP)$ is simple. If $m$ is even, when $m > 0$ or $m = 0, n \ge 2$, $\mfrak{sl}(m \cdot \1 + nP)$ is an extension of $\mfrak{psl}(m \cdot \1 + nP) := \mfrak{sl}(m \cdot \1 + nP)/\1 I_{m+nP}$ by $\1$, and $\mfrak{psl}(m \cdot \1 + nP)$ is simple.
\end{proposition}
\begin{proof}
    We can show this directly. 
    Let
    \begin{equation*}
        \mfrak{g} = \begin{cases}
            \mfrak{sl}(m \cdot \1 + nP) & m \text{ odd} \\
            \mfrak{psl}(m \cdot \1 + nP) & m \text{ even}.
        \end{cases}
    \end{equation*}
    Suppose $X \in \mfrak{g}$. For $z, w \in S$, let $X_{z,w}$ be the $(z, w)$ matrix coefficient of $X$ with respect to the basis $S$. Then bracketing with an off-diagonal basis element of $\mfrak{g}$ gives us a matrix that is nonzero in at most 2 rows and 1 column:
    \begin{itemize}
        \item $Y := [X, x_i \otimes x_j^*]$ has $Y_{x_i,z} = X_{x_j,z}$, $Y_{z, x_j} = X_{z,x_i}$. $Y_{x_i, x_j} = X_{x_i, x_i} + X_{x_j, x_j}$. All other entries of $Y$ are 0.
        \item $Y := [X, x_i \otimes y_j^*]$ has $Y_{x_i, z} = X_{y_j, z}$, $Y_{z, y_j} = X_{z, x_i}$. Again, where these intersect, we have $Y_{x_i, y_j} = X_{x_i, x_i} + X_{y_j, y_j}$.
        \item $Y := [X, x_i \otimes {y'_j}^*]$ has $Y_{x_i, z} = X_{y'_j, z}$ if $z \ne y_k$ and $Y_{x_i, y_k} = X_{y'_j, y_k} + X_{y_j, y'_k}$, and $Y_{z, y'_j} = X_{z, x_i}$.
        \item $Y := [X, y'_i \otimes x_j^*]$ has $Y_{y'_i,z} = X_{x_j, z}$, $Y_{z, x_j} = X_{z, y'_i}$.
        \item $Y:= [X, y_i \otimes x_j^*]$ has $Y_{y_i, z} = X_{x_j, z}$, $Y_{z, x_j} = X_{z, y_i}$, and $Y_{y_i', y_k} = X_{x_j, y'_k}$
        \item $Y := [X, y_i \otimes y_j^*]$ has $Y_{y_i, z} = X_{y_j, z}$, $Y_{z, y_j} = X_{z, y_i}$, and $Y_{y'_i, y_k} = X_{y_j, y'_k}$. 
        \item $Y := [X, y'_i \otimes y_j^*]$ has $Y_{y'_i, z} = X_{y_j, z}$ and $Y_{z, y_j} = X_{z, y'_i}$.
        \item $Y := [X, y'_i \otimes {y'_j}^*]$ has $Y_{y'_i, z} = X_{y'_j, z}$ when $z \ne y_k$ and $Y_{y'_i, y_k} = X_{y'_j, y_k} + x_{y_j, y'_k}$; and $Y_{z, y'_j} = X_{z, y'_i}$.
        \item $Y := [X, y_i \otimes {y'_j}^*]$ has $Y_{y_i, z} = X_{y'_j, z}$ when $z \ne y_k$ and $Y_{y_i, y_k} = X_{y'_j, y_k} + X_{y_j, y'_k}$, $Y_{z, y'_j} = X_{z, y_i}$, $Y_{y'_i, y'_k} = X_{y_j, y'_k}$, $Y_{y'_i, y_k} = X_{y_j y_k} + X_{y'_j, y'_k}$.
    \end{itemize}
    Suppose that $X \in J$ some nontrivial ideal of $\mfrak{g}$.
    If $X$ has a nonzero off-diagonal entry, we can iteratively bracket with basis elements when $m, n \ge 1$ until we get a multiple of a basis element. For example, suppose that $X_{y_i \otimes x_j^*} \ne 0$. Then 
    \begin{equation*}
       [[[X, x_j \otimes y_i^*], y'_i \otimes x_j^*], y_i \otimes {y'_i}^*] = X_{y_i, x_j}(y_i \otimes x_j^*).
    \end{equation*}
    Hence $y_i \otimes x_j^* \in J$, and further brackets with other basis elements produce the whole basis of $\mfrak{g}$, so $J = \mfrak{g}$. The argument for the other off-diagonal elements is analogous.
    
    If $X$ is diagonal and has two distinct elements on the diagonal, e.g. $X_{z, z}$ and $X_{w, w}$, then
    \begin{equation*}
        [X, z \otimes w^*] = (X_{z, z} - X_{w, w})(z \otimes w^*),
    \end{equation*}
    and then we can again take brackets with other basis elements to get that $J = \mfrak{g}$. When $m$ is odd, no traceless diagonal matrix can be a multiple of the identity, so this covers all cases for $X$. When $m$ is even, since we have quotiented out by multiples of the identity, this covers all cases.

    Therefore, $\mfrak{g}$ is simple when $m, n \ge 1$.
\end{proof}
Recall that if a Lie algebra $\mfrak{g}$ is simple, then there is only one adjoint-invariant bilinear form up to scalars. 
\begin{corollary}
    Therefore, when $m$ is odd, $\mfrak{sl}(m \cdot \1 + nP)$ has exactly one adjoint-invariant bilinear form up to scalars.
\end{corollary}
\begin{remark}
Note that the conditions for $\mfrak{sl}(m \cdot \1 + nP)$ to be simple are different from when the Lie superalgebra $\mfrak{sl}(m \cdot \1 + n | n)$ is simple: $\mfrak{sl}(m \cdot \1 + n | n)$ is simple when $m, n \ge 1$, and when $m = 0$, $\mfrak{psl}(n | n) := \mfrak{sl}(n | n)/kI$ is simple.
\end{remark}

When $m$ is even, we have only deduced that $\mfrak{psl}(m \cdot \1 + nP)$ has one adjoint-invariant bilinear form, but we can extend it to $\mfrak{sl}(m \cdot \1 + nP)$ uniquely.
\begin{corollary}
\label{sl_m_even_form}
    When $m$ is even, $\mfrak{sl}(m \cdot \1 + nP)$ has exactly one adjoint-invariant bilinear form up to scalars.
\end{corollary}
\begin{proof}
Let $B$ be an adjoint-invariant bilinear form on $\mfrak{sl}(m \cdot \1 + nP)$.
Notice that $I = \sum_{z \in S} z \otimes z^*$ (corresponding to the identity matrix) is central. Because
\begin{equation*}
    \left[\sum_{i = 1}^{m/2} x_{2i - 1} \otimes x_{2i}^* + \sum_{i = 1}^{n} y_i \otimes {y'_i}^*, \sum_{i = 1}^{m/2} x_{2i} \otimes x_{2i - 1}^* + \sum_{i = 1}^{n} y'_i \otimes y_i^*\right] = I
\end{equation*}
and $\mfrak{psl}$ is simple, $[\mfrak{sl}, \mfrak{sl}] = \mfrak{sl}$. Hence, $B(I, g) = B(g, I) = 0$ for all $g \in \mfrak{sl}(m \cdot \1 + nP)$.
Then $B$ is determined by its values on the quotient $\mfrak{psl}(m \cdot \1 + nP)$, which we know has a unique adjoint-invariant bilinear form up to scalars.
\end{proof}

\begin{proposition}
    $\mfrak{gl}(m \cdot \1 + nP)$ has a two-dimensional space of invariant bilinear forms.
\end{proposition}
\begin{proof}
    If $m$ is odd, $\mfrak{gl}(m \cdot \1 + nP) = \1 \oplus \mfrak{sl}(m \cdot \1 + nP)$, so picking an invariant bilinear form $B$ is equivalent to defining it on $\1$ and on $\mfrak{sl}(m \cdot \1 + nP)$.

    If $m$ is even, $\mfrak{gl}(m \cdot \1 + nP)$ is an extension of $\1$ by $\mfrak{sl}(m \cdot \1 + nP)$; let $w \in \mfrak{gl}(m \cdot \1 + nP)$ be a lift of the generator of $\1$. $\mfrak{sl}(m \cdot \1 + nP)$ is an ideal, so for $g \in \mfrak{sl}(m \cdot \1 + nP)$, $[w, g] \in \mfrak{sl}(m \cdot \1 + nP)$. Let $B$ be an invariant bilinear form on $\mfrak{gl}$; $B$ is determined up to scalars on $\mfrak{sl}$. For $g \in \mfrak{sl}$, invariance of $B$ implies that $B(w, g) = B(w, [g_1, g_2]) = B([g_1, w], g_2)$, so $B(w, \mfrak{sl})$ is fixed by values of $B$ on $\mfrak{sl}$. On the other hand, $B(w, w)$ can be chosen arbitrarily because there do not exist $g_1, g_2 \in \mfrak{gl}$ such that $[g_1, g_2] = w$. So $B$ is determined by two parameters.
\end{proof}

\begin{proposition}
    If $m$ is odd, the Casimir element of $\mfrak{gl}(m \cdot \1 + nP)$ is
    \begin{equation*}
    \frac{1}{\lambda} \sum_{i = 1}^m ((x_i \otimes x_i^*)^2 - x_i \otimes x_i^*) + \frac{1}{\lambda} \sum_{z \in S} z \otimes z^*.
    \end{equation*}
\end{proposition}
\begin{proof}
    Let $B$ be an invariant bilinear form on $\mfrak{gl}(m \cdot \1 + nP) = \1 \oplus \mfrak{sl}(m \cdot \1 + nP)$. $B$ restricted to $\mfrak{sl}(m \cdot \1 + nP)$ is a multiple of the trace form, say $B(x, y) = \mu \Tr(xy)$, and $B(I, I) = \lambda$ for some other parameter $\lambda$. Let a basis for $\mfrak{sl}$ be $e_{zw} = z \otimes w^*, z \ne w$ and $h_z = z \otimes z^* + y'_n \otimes {y'_n}^*, z \ne y'_n$. Let $I = \sum z \otimes z^*$ be the generator of $\1$. Then the isomorphism $\mfrak{g} \cong \mfrak{g}^*$ induced by $B$ is defined by $I \mapsto \lambda I^*$, $e_{zw} \mapsto \mu e_{wz}^*$, and $h_z \mapsto \mu \sum_{w \ne z} h_w^*$. Its inverse map $\mfrak{g}^* \cong \mfrak{g}$ is defined by
    \begin{align*}
        I^* &\mapsto \frac{1}{\lambda} I \\
        e_{zw}^* &\mapsto \frac{1}{\mu} e_{wz} \\
        h_z^* &\mapsto \frac{1}{\mu} \sum_{w \ne z} h_w^*.
    \end{align*}
    Hence the Casimir element is
    \begin{equation*}
        \frac{1}{\mu} \left(\sum_{z, w \in S, z \ne w} e_{zw} e_{wz} + \sum_{z \in S} h_z \sum_{z \ne w} h_w\right) + \frac{1}{\lambda} I^2
    \end{equation*}
    and $[e_{zw}, e_{wz}] = z \otimes z^* + w \otimes w^*$ while $[h_z, h_w] = 0$, so this simplifes to
    \begin{equation*}
        \frac{1}{\mu}\left((m + 2n - 1)\sum_{z \in S} (z \otimes z^*)\right) + \frac{1}{\lambda} I^2.
    \end{equation*}
    Because $m$ is odd, $m + 2n - 1$ is even, so the first summand vanishes. $I^2$ further simplifies because $[y_i \otimes {y'_i}^*, y_i \otimes {y'_i}^*] = y_i \otimes y_i^* + y'_i \otimes {y'_i}^*$, so in $U(\mfrak{gl})$, $y_i \otimes y_i^* + y'_i \otimes {y'_i}^*$ is idempotent, giving us the expression in the proposition.
\end{proof}

\begin{proposition}
    If $m$ is even, the Casimir element of $\mfrak{gl}(m \cdot \1 + nP)$ is
    \begin{equation*}
        \frac{\lambda}{\mu^2}\left(\sum_{i = 1}^m (x_i \otimes x_i^*)^2 + x_i \otimes x_i^*\right) + \frac{\lambda + \mu}{\mu^2}\left(\sum_{z \in S} z \otimes z^*\right).
    \end{equation*}
\end{proposition}
\begin{proof}
    Let $B$ be an invariant bilinear form on $\mfrak{gl}(m \cdot \1 + nP)$, which is an extension of $\1$ by $\mfrak{sl}(m \cdot \1 + nP)$; we can lift the generator of $\1$ to any matrix in $\mfrak{gl}$ with nonzero trace, WLOG $y_n \otimes y_n^* =: h_{y_n}$. Again, $B$ restricted to $\mfrak{sl}(m \cdot \1 + nP)$ is a multiple of the trace form, say $B(x, y) = \mu \Tr(xy)$. To define $B$ on all of $\mfrak{gl}$, we additionally need to set $B(h_{y_n}, -)$. Let $I = \sum z \otimes z^*$ be the identity matrix, $e_{zw} = z \otimes w^*, z \ne w$, and $h_z =  z\otimes z^* + y'_n \otimes {y'_n}^*$ for $z \ne y_n, {y'_n}^*$, so $e_{zw}, h_z, I$ form a basis of $\mfrak{sl}$, and along with $h_{y_n}$, they form a basis of $\mfrak{gl}$.
    
    From Corollary \ref{sl_m_even_form}, $B(h_{y_n}, h_{y_n}) = \lambda$ for some $\lambda \in k$, and we can compute (using adjoint-invariance of $B$ and that $\mfrak{sl}$ is a Lie ideal) that
    \begin{align*}
        B(h_{y_n}, I) = B(I, h_{y_n}) &= \mu \\
        B(h_{y_n}, h_z) = B(h_z, h_{y_n}) &= 0 \\
        B(h_{y_n}, e_{zw}) = B(e_{zw}, h_{y_n}) &= 0.
    \end{align*}
     Hence, the isomorphism $\mfrak{g}^* \cong \mfrak{g}$ induced by $B$ sends
    \begin{align*}
        e_{wz}^* &\mapsto \frac{1}{\mu}e_{zw} \\
        h_z^* &\mapsto \frac{1}{\mu} \sum_{w \ne z, y_n, y'_n} h_w \\
        I^* &\mapsto \frac{1}{\mu} h_{y_n} - \frac{\lambda}{\mu^2} I \\
        h_{y_n}^* &\mapsto \frac{1}{\mu}I.
    \end{align*}
    The Casimir element is then 
    \begin{equation*}
        \frac{1}{\mu}(m + 2n-1) \sum_{z \in S} (z \otimes z^*) + \frac{\lambda}{\mu^2} I^2.
    \end{equation*}
    By the same reasoning as the odd $m$ case, this simplifies to the expression in the statement of the proposition.
\end{proof}

\section{Representations of $\mfrak{gl}(P)$}
In this section, we classify the irreducible representations of $\mfrak{gl}(P)$ and relate these to the irreducible representations of $\GL(P)$ discussed in \cite{hu_supergroups_2024}.

\subsection{$\mfrak{gl}(P)$ and $U(\mfrak{gl}(P))$}

We first describe $\mfrak{gl}(P)$, $U(\mfrak{gl}(P))$, and the center of $U(\mfrak{gl}(P))$. The Lie algebra $\mathfrak{g}\mathfrak{l} (P)$ is $P \otimes P^{*} =
\underline{\End} (P)$, and the commutator is
\begin{equation*} [A, B] = AB - BA - A'B'. \end{equation*}
The representations of $\mfrak{gl}(P)$ correspond to representations of its universal enveloping algebra. $U(\mfrak{g})$ satisfies the PBW theorem, so let us find a basis for it:

If $P$ has basis $1, d$ and $P^*$ has dual basis $1^*, d^*$, then $\mfrak{gl}(P)$ has basis $1 \otimes 1^*$, $1 \otimes d^*$, $d \otimes 1^*$, $d \otimes d^*$. Note that $d(d^*) = 1^*$ and $d(1^*) = 0$. The $d$-action on $\mfrak{gl}(P)$ acts by
\begin{align*}
    1 \otimes 1^*, d \otimes d^* &\mapsto d \otimes 1^* \mapsto 0 \\
    1 \otimes d^* &\mapsto d \otimes d^* + 1 \otimes 1^*.
\end{align*}
We will instead use the basis $x = 1 \otimes d^*, x' = e = 1 \otimes 1^* + d \otimes d^*, y = 1 \otimes 1^*, y' = d \otimes 1^*$ because $e$ is both central (it is the Casimir element here) and idempotent, as $(x')^2 = [x, x] = x'$ in $U(\mfrak{g})$. Moreover, $(y')^2 = 0$ since $(y')^2 = [y, y] = 0$.

From the discussion of the ``$4$-center'' in Section \ref{pcenter}, we additionally obtain that the elements $x^2$, $y^4 - y^2$ are central.

Since $e$ is a central idempotent, we can write $U (\mathfrak{g}) = U
(\mathfrak{g}) e \oplus U (\mathfrak{g}) (1 - e)$. Set $A_1 \assign U (\mathfrak{g}) e$ and $A_0
\assign U (\mathfrak{g}) (1 - e)$.

\begin{proposition}
    The center of $U(\mfrak{gl}(P))$ is generated by $e$, $x^2$, $y^4 + y^2$, $(y^2 + y)e$, and $xy'(1 - e)$.
\end{proposition}
\begin{proof}
    It suffices to find central elements in $A_1$ and $A_0$ separately.
    \begin{lemma}
        The center of $A_1$ is generated by $x^2$ and $y^2 + y$.
    \end{lemma}
    \begin{proof}
    In $A_1$, we have
    \begin{align*}
        xy' + y'x &= 1 \\
        yy' + y'y &= y' \\
        xy + yx &= x.
    \end{align*}
    We will compute brackets of monomials with $x, y, y'$; since $(y')^2 = 0$ and $x^2$ is central, we only need to consider monomials with $y'$-degree and $x$-degree at most $1$. First, we have
    \begin{equation*}
        y^b x = x(y + 1)^b,
    \end{equation*}
    so
    \begin{align*}
        x^ay^by'x + x^{a + 1}y^b y' &= x^{a+1}(y+1)^by' + x^a y^b + x^{a+1}y^by' = 
            x^{a+1}((y+1)^b + y^b)y' + x^a y^b \\
        x^a y^b x + x^{a+1}y^b &= x^{a+1}(y+1)^b + x^{a+1}y^b = 
            x^{a+1}((y+1)^b + y^b)
    \end{align*}
    Hence, we see that the only linear combinations of monomials that commute with $x$ are those generated by $y^2 + y$ and $x$. 
    Next, we have
    \begin{equation*}
        x^a y + y x^a = a x^a
    \end{equation*}
    so central elements must have even $x$-degree.
    Finally, we can check that 
    \begin{align*}
        x^a y' + y' x^a &= a x^{a-1} \\
        y^by' &= y'(y+1)^b
    \end{align*}
    so all elements in $\langle x^2, y^2 + y \rangle$ are central. Therefore, the center of $A_1$ is generated by $x^2$ and $y^2 + y$.
\end{proof}
\begin{lemma}
    The center of $A_0$ is generated by $x^2$, $xy'$, and $y^4 + y^2$.
\end{lemma}
\begin{proof}
    In $A_0$, we have
    \begin{align*}
        xy' + y'x &= 0 \\
        yy' + y'y &= y' \\
        xy + yx &= x + y'.
    \end{align*}
    Again, we first consider monomials $x^a y^b$ and $x^a y^b y'$. We first consider how they bracket with $y'$; since $x$ commutes with $y'$, it suffices to consider $[y^b, y']$. Like in the $A_1$ case,
    \begin{equation*}
        y^by' = y'(y+1)^b,
    \end{equation*}
    so the centralizer of $y'$ is generated by $x$, $y'$, and $y^2 + y$. Now consider whether these commute with $y$:
    \begin{align*}
        x^ay &= yx^a + a(x^a + x^{a-1}y') \\
        \implies x^a(y^{2^k} + y)y + yx^a (y^{2^k} + y)y &= a(x^a + x^{a-1}y')(y^{2^k} + y), \\
        x^ay'y + yx^a y' &= x^a y y' + + x^a y' + x^a y y' + a(x^a + x^{a-1}y')y' = (a+1)x^a y'
    \end{align*}
    and
    \begin{align*}
         &\text{ }x^a(y^{2^k} + y)y'y + yx^a(y^{2^k} + y)y' \\
        = &\text{ }x^a(y^{2^k} + y)yy' + x^a(y^{2^k} + y)y' + x^a y(y^{2^k} + y)y' + a(x^a + x^{a-1}y')(y^{2^k} + y)y' \\
        = &\text{ }(a + 1)x^a(y^{2^k} + y)y'.
    \end{align*}
    Therefore, an element commuting with both $y'$ and $y$ must be in $\langle x^2, xy', y^2 + y \rangle$. 
    Finally, we can check that 
    \begin{align*}
        yx + xy &= x + y' \\
        y^{2^k} x + xy^{2^k} &= x(y + 1)^{2^k} + (2^k)(y^{2^k} + y') + xy^{2^k} = x 
    \end{align*}
    so an element commuting with all three of $x, y', y$ must be in $\langle x^2, xy', y^4 + y^2 \rangle$.
    
    Thus the center of $A_0$ are generated by $x^2$, $xy'$, $y^4 + y^2$.
    \end{proof}
    Therefore, the center of $U(\mfrak{gl}(P))$ is generated by $x^2$, $y^4 + y^2$, $(y^2 + y)e$, and $xy'(1 - e)$.
\end{proof}

\begin{remark}
This implies that the ``$4$-center'' ($x^2, y^4 + y^2$) and Casimir element ($e$) do not generate the center of $U(\mfrak{g})$ since $(y^2 + y)e$ and $xy'(1 - e)$ do not lie in the algebra they generate.
\end{remark}

\begin{question}
    How can we describe the whole center of $U(\mfrak{gl}(m \cdot \1 + nP))$?
\end{question}

\subsection{Classification of irreducible representations of $\mfrak{gl}(P)$}\label{glp_irreps}
In this section, we find all irreducible representations of $\mfrak{gl}(P)$. Suppose $M$ is a simple $U
(\mathfrak{g})$-module. Then since $e$ is a central idempotent, either $e$ acts by $1$ on $M$, so $M$ is a simple $U
(\mathfrak{g}) e$-module, or it acts by 0, so $M$ is a simple $U
(\mathfrak{g}) (1 - e)$-module. Hence, it suffices to find simple $A_0$-modules and simple $A_1$-modules.
\begin{proposition}
     The simple modules of $A_1$ are a 2-dimensional family of modules with underlying object $P$ parametrized by $a, b \in k$; call these $L(1, a, b)$. The $\mathfrak{gl}(P)$-action on $L(1, a,b)$ is
     \begin{equation*} y \mapsto \begin{pmatrix}
     b & 0\\
     0 & b + 1
   \end{pmatrix}, x \mapsto \begin{pmatrix}
     0 & 1\\
     a & 0
   \end{pmatrix}, y' \mapsto \begin{pmatrix}
     0 & 0\\
     1 & 0
   \end{pmatrix}. \end{equation*}
\end{proposition}
\begin{proof}
In simple $A_1$-modules, $x^2$ and $y^2 - y$ will act by scalars; suppose $x^2$ acts by $a^2$ and $y^2 - y$ by $b(b+1)$. Then the quotient $B_1 := A_1/(x^2 - a, y^2 - y - b(b+1))$ is 8-dimensional. $B_1$ has two simple 2-dimensional modules: one is isomorphic to the submodule of $B_1$ spanned by $bxy' + xyy', by' + yy'$, and
the other to the submodule spanned by $(b + 1) xy' + xyy', (b + 1) y' + yy'$.

On the former, the action is
\begin{equation*} y \mapsto \begin{pmatrix}
     b & 0\\
     0 & b + 1
   \end{pmatrix}, x \mapsto \begin{pmatrix}
     0 & 1\\
     a & 0
   \end{pmatrix}, y' \mapsto \begin{pmatrix}
     0 & 0\\
     1 & 0
   \end{pmatrix} \end{equation*}
and on the latter it's
\begin{equation*} y \mapsto \begin{pmatrix}
     b + 1 & 0\\
     0 & b
   \end{pmatrix}, x \mapsto \begin{pmatrix}
     0 & 1\\
     a & 0
   \end{pmatrix}, y' \mapsto \begin{pmatrix}
     0 & 0\\
     1 & 0
   \end{pmatrix} . \end{equation*}
Distinct values of $a, b$ give distinct submodules.
\end{proof}

\begin{remark}
  $B_1$ is semisimple since we can check that any element annihilating both
  modules is 0, so it's the product of two matrix algebras of size 2.
\end{remark}

We now find the simple $A_0$-modules.
\begin{proposition}
    The element $xy'$, which is central in $A_0$, acts by $0$ on a simple $A_0$-module.
\end{proposition}
\begin{proof}
    We proved above that $xy'$ is central in $A_0$. It is nilpotent since $(y')^2 = 0$, so on simple modules it must act by $0$.
\end{proof}

As mentioned previously, $x^2$ and $y^4 + y^2$ are central in $A_0$ as well.
\begin{proposition}
    The simple $A_0$-modules form a two-dimensional family parametrized by scalars $a, b \in k$; call them $L(0, a, b)$.
    \begin{itemize}
        \item If $a = 0$, $L(0, 0, b)$ is one-dimensional. The $\mfrak{gl}(P)$ action is 
        \begin{equation*}
            y \mapsto (b), x, y' \mapsto (0).
        \end{equation*}
        \item If $a \ne 0$, $L(0, a, b)$ is two-dimensional with underlying object $\1^2$. The $\mfrak{gl}(P)$-action is 
        \begin{equation*}
            y
  \mapsto \begin{pmatrix}
    b & 0\\
    0 & b + 1
    \end{pmatrix}, x \mapsto \begin{pmatrix}
    0 & 1\\
    a & 0
  \end{pmatrix}, y' \mapsto (0).
        \end{equation*}
    \end{itemize}
    Note that if $a = 1$, $L(0, 1, b) = L(0, 1, b + 1)$.
\end{proposition}
\begin{proof}
Since $x^2, y^4 - y^2$ must act by scalars, choose $a, b$ such that $x^2$ acts by $a^2$ and $y^4 - y^2$ acts by $b^2(b^2 + 1)$. 
Let $B_0 \assign A_0 / (x^2 - a^2, xy', y^4 - y^2 - b^2 (b^2 + 1))$. 
    An explicit basis of $B_0$ is $1, y, y^2, y^3$, $x, xy, xy^2, xy^3$,
$y', yy', y^2 y', y^3 y'$.
Let $z_1 := b$ and $z_2 := b + 1$.

Consider the basis of $B_0$ given by 
\begin{align*}
  X_1 & = z_1 z_2^2 + z_2^2 y + z_1 y^2 + y^3 \\
  X_2 & = z_1^2 z_2 + z_1^2 y + z_2 y^2 + y^3 \\
  \bar{X}_1 & = z_2^2 + y^2 \\
  \bar{X}_2 & = z_1^2 + y^2\\
  Y_1 & = z_1 z_2^2 y' + z_2^2 yy' + z_1 y^2 y' + y^3 y' = X_1 y'\\
  Y_2 & = X_2 y'\\
  \bar{Y}_i & = \bar{X}_i y' \\
  x \overline{X_i} & \\
  xX_i &
\end{align*}

The actions of $x, y, y'$ on $X_i, Y_i, \overline{X_i}, \overline{Y_i}, x \overline{X_i}, x X_i$ are
\begin{align*}
  xY_i = 0, &\quad x (xX_i) = a X_i, \\
  yY_i = z_i Y_i, &\quad yX_i = z_i X_i, \quad y (xX_i) = z_{i + 1}  (xX_i) + Y_{i + 1}, \\
  y' Y_i = 0 &\quad y' X_i = Y_{i + 1}, \quad y' (xX_i) = 0, \\
  x \overline{Y_i} = 0, &\quad x (x \overline{X_i}) = a^2  \overline{X_i}, \\
  y \overline{Y_i} = Y_i + z_i  \overline{Y_i},  &\quad y \overline{X_i} = X_i + z_i  \overline{X_i},\quad  y (x \overline{X_i}) = xX_i + z_{i + 1} x \overline{X_i} + \overline{Y_{i + 1}}, \\
  y'  \overline{Y_i} = 0, &\quad y'  \bar{X}_i = \overline{Y_{i + 1}}, \quad y' (x \overline{X_i}) = 0.
\end{align*}
Therefore, the $Y_i$ are 1-dimensional submodules where $y$ acts by $z_i$, $x,
y'$ act by 0; call these $M_i$. The 4-dimensional submodule spanned by $Y_i,
\overline{Y_i}$ is a direct sum of (nontrivial extension of 2 copies of $M_1$,
spanned by $Y_1, \overline{Y_1}$) and (nontrivial extension of 2 copies of
$M_2$, spanned by $Y_2, \overline{Y_2}$).

Let $C$ be the quotient of $B_0$ by this 4-dimensional submodule. Then for each $i \in \{1, 2\}$, $X_i, xX_i$ form a
2-dimensional submodule of $C$ killed by $y'$ and
\[ y \mapsto \left(\begin{array}{cc}
     z_i & 0\\
     0 & z_{i + 1}
   \end{array}\right), x \mapsto \left(\begin{array}{cc}
     0 & a\\
     1 & 0
   \end{array}\right) \]
call this $N_i$. We can check that $C$ is a direct sum of (nontrivial extension of 2 copies of $N_1$
spanned by $X_1, \overline{X_1}$, and $x$ times these) and (nontrivial
extension of 2 copies of $N_2$, spanned by the $X_2$'s, etc.). So we get a
2-dimensional family of simple modules parametrized by $a, b$. If $a = 0$,
then both $N_i$ are nontrivial extensions of $M_1, M_2$. If $a = 1$, $N_1 = N_2$.
\end{proof}

\subsection{Tensor products of the simple $\mathfrak{gl}(P)$-representations}

\begin{proposition}
    The tensor products of simple $\mfrak{gl}(P)$ representations decompose as follows (all extensions are nontrivial, the head is first):
    \begin{itemize}
    \item $L(1, a_1, b_1) \otimes L(1, a_2, b_2) = [L(0, a_1 + a_2, b_1 + b_2), L(0, a_1 + a_2, b_1 + b_2 + 1)]$
    \item $L(0, a_1, b_1) \otimes L(0, a_2, b_2) = [L(0, a_1 + a_2, b_1 + b_2), L(0, a_1 + a_2, b_1 + b_2 + 1)]$
    \item $L(0, 0, b_1) \otimes L(0, a, b_2) = L(0, a, b_1 + b_2)$
    \item $L(0, 0, b_1) \otimes L(1, a, b_2) = L_1 (a, b_1 + b_2)$
    \item $L(0, a_1, b_1) \otimes L(1, a_2, b_2) = [L (1, a_1 + a_2, b_1 + b_2), L(1, a_1 + a_2, b_1 + b_2 + 1)]$.
\end{itemize}
\end{proposition}
\begin{proof}
Tensoring $L(1, a_1, b_1)$ with $L(1, a_2, b_2)$, we get a 4-dimensional module
where $y$ acts by $\diag (b_1 + b_2, b_1 + b_2 + 1, b_1 + b_2 + 1, b_1 + b_2)$,
\begin{equation*} y' \mapsto \begin{pmatrix}
     0 & 1 & 1 & 0\\
     0 & 0 & 0 & 1\\
     0 & 0 & 0 & 1\\
     0 & 0 & 0 & 0
   \end{pmatrix}, x \mapsto \begin{pmatrix}
     0 & a_2 & a_1 & 0\\
     1 & 0 & 0 & a_1\\
     1 & 0 & 0 & a_2\\
     0 & 1 & 1 & 0
   \end{pmatrix}, x' \mapsto 0. \end{equation*}
Hence, this is an extension of $L(0, a_1 + a_2, b_1 + b_2)$ by $L(0, a_1 + a_2, b_1
+ b_2 + 1)$. We can likewise compute the others and conclude the extensions are nontrivial.

Note that 
\begin{equation*}
    L(1, a_1, b_1) \otimes L(1, a_2, b_2) \not\cong L(0, a_1, b_1) \otimes L(0, a_2, b_2),
\end{equation*}
though they are extensions of the same two representations: in the former $y'$ does not act trivially, while it does in the latter. 
\end{proof}

\subsection{Restriction of $\GL(P)$-representations to $\mfrak{gl}(P)$-representations}
In this section, we relate the irreducible $\GL(P)$-representations, as classified in \cite{hu_supergroups_2024}, to the irreducible $\mfrak{gl}(P)$-representations classified above.
\begin{proposition}
    The irreducible representations of $\GL(P)$ restrict to $\mfrak{gl}(P)$-representations as follows:
    \begin{itemize}
        \item Tensor powers of $\chi$ restrict to the trivial $\mfrak{gl}(P)$ representation
        \item $T_1 \otimes \chi^n$ reduces to $L(1, 0, 0)$
        \item $T_2 \otimes \chi^n$ reduces to $L(0, 1, 0) = L(0, 1, 1)$
        \item $T_3 \otimes \chi^n$ reduces to $L(1, 1, 0)$
        \item $\xi \otimes \chi^n$ reduces to $L(0, 0, 1)$
        \item $\xi T_1 \otimes \chi^n$ reduces to $L(1, 0, 1)$
        \item $\xi T_3 \otimes \chi^n$ reduces to $L(1, 1, 1)$
    \end{itemize}
\end{proposition}
\begin{proof}
The Lie algebra of $\GL(P)$ is $(I_1/I_1^2)^*$, $I_1 \subset k[G]$ is the augmentation ideal.
Explicitly, since elements of $P$ are of 
the form $\begin{pmatrix}
  a & a'\\
  b & a + b'
\end{pmatrix}$ where $a$ is invertible,
\begin{equation*}
    k [G] = k [A - 1, B, A^{- 1} -
1, A', B'], I_1 = (A - 1, B, A', B').
\end{equation*}
The Lie algebra of $G$ consists of $\mu \in
\Dist (G)$ with $\mu (1) = 0, \mu (I_1^2) = 0$; the space of such $\mu$ is spanned by $\mu_{A - 1}, \mu_B, \mu_{A'}, \mu_{B'}$; this is isomorphic to
$\mathfrak{gl} (P)$ as above with $x \mapsto \mu_{A'}$, $y \mapsto
\mu_{B'}$, $x' \mapsto \mu_{A - 1}$, $y \mapsto \mu_B$.

Then the $\GL(P)$ representations have the following $\mfrak{gl}(P)$-action:
\begin{itemize}
    \item Trivial representation: The $\GL(P)$-coaction is $v \mapsto 1 \otimes v$, so everything in
$\Lie (G)$ acts by 0. Therefore, it is $L(0, 0, 0)$ as a $\mfrak{gl}(P)$-representation.
\item $T_1$: the $\GL(P)$ coaction is
\begin{align*}
    v &\mapsto (A - 1) \otimes v + 1 \otimes v + B \otimes
v' \\
v' &\mapsto A' \otimes v + (A - 1) \otimes v' + 1 \otimes v' + B'
\otimes v'.
\end{align*}
So $\mu_{A'} \mapsto \begin{pmatrix}
  0 & 1\\
  0 & 0
\end{pmatrix}$, $\mu_{B'} \mapsto \begin{pmatrix}
  0 & 0\\
  0 & 1
\end{pmatrix}$, $\mu_A \mapsto I$, and $\mu_B \mapsto
\begin{pmatrix}
  0 & 0\\
  1 & 0
\end{pmatrix}$. So this is the simple 2-dimensional $A_1$-module with $a =
0, b = 0$, i.e. $L(1, 0, 0)$.
\item $T_2$: the $\GL(P)$ coaction is
\begin{eqnarray*}
  v & \mapsto & (A - 1)^2 \otimes v + 1 \otimes v\\
  &  & + A' B' \otimes w + A' \otimes w + (A - 1) A' \otimes w\\
  w & \mapsto & A' \otimes v + (A - 1) A' \otimes v\\
  &  & + (A - 1)^2 \otimes w + 1 \otimes w + B' \otimes w + (A - 1) B'
  \otimes w + BA' \otimes w
\end{eqnarray*}
with $v' = t w = 0$, so the $\Lie (G)$ action is
\begin{equation*} \mu_{A'} : \begin{pmatrix}
     0 & 1\\
     1 & 0
   \end{pmatrix}, \mu_{B'} : \begin{pmatrix}
     0 & 0\\
     0 & 1
   \end{pmatrix}, \mu_{A - 1} = \mu_B = 0. \end{equation*}
So this is a 2-dimensional simple $A_0$-module with $a = 1$, $b = 0$ or $1$, i.e. $L(0, 1, 0)$, which is the same as $L(0, 1, 1)$.
\item $T_3$: the $\GL(P)$ coaction is
\begin{eqnarray*}
  v & \mapsto & ((A - 1)^3 + (A - 1)^2 + (A - 1) + 1) \otimes v + (A - 1)^2 B
  \otimes v' + B \otimes v'\\
  &  & + (A - 1) A' B' \otimes v' + A' B' \otimes v' + A' \otimes v' + (A
  - 1)^2 A' \otimes v'
\end{eqnarray*}
so
\begin{equation*} \mu_{A'} : \begin{pmatrix}
     0 & 1\\
     1 & 0
   \end{pmatrix}, \mu_{B'} : \begin{pmatrix}
     0 & 0\\
     0 & 1
   \end{pmatrix}, \mu_{A - 1} : I, \mu_B : \begin{pmatrix}
     0 & 0\\
     1 & 0
   \end{pmatrix} . \end{equation*}
This is the 2-dimensional $A_1$-module with $a = 1, b = 0$, i.e. $L(1, 1, 0)$.
\item $\chi$ also restricts to the trivial $\mathfrak{gl}(P)$-representation.
\item $\xi$ (where $g$ acts by $a^4 +
a^3 b' + a^2 ba'$) has $\mu_{B'}$ acting by 1, while the other elements act by 0. So this is the 1-dimensional $A_0$-module with $b = 1$, i.e. $L(0, 0, 1)$.
\end{itemize}

The other representations of $\GL (P)$ are tensor products of the above.
In particular, tensoring by $\xi$ sends $b \mapsto b + 1$, so this sends $T_2$
to itself but $T_1, T_3$ to something different.
\end{proof}
\begin{corollary}
    The irreducible $\GL(P)$ representations restrict to exactly the irreducible $\mfrak{gl}(P)$ representations with parameters in $\mbb{F}_2$.
\end{corollary}

\printbibliography
\end{document}